\documentclass[reqno,11pt]{amsart}

\usepackage{amsmath}
\usepackage{amsthm}
\usepackage{amssymb}
\usepackage{amscd}
\usepackage{xypic}
\usepackage{verbatim}
\usepackage[plainpages=false,colorlinks,hyperindex,pdfpagemode=None,bookmarksopen,linkcolor=red,citecolor=blue,urlcolor=blue]{hyperref}
\usepackage{pdflscape}
\usepackage{stmaryrd}
\usepackage{yfonts} 
\usepackage{palatino}
\usepackage{enumerate}
\usepackage{extarrows}

\usepackage{bbm}

\def\beq{\begin{equation}}
\def\eeq{\end{equation}}

\swapnumbers
\theoremstyle{plain}
\newtheorem{theorem}[subsubsection]{Theorem}

\newtheorem*{theorem*}{Theorem}
\newtheorem{proposition}[subsubsection]{Proposition}
\newtheorem*{proposition*}{Proposition}
\newtheorem{lemma}[subsubsection]{Lemma}
\newtheorem*{lemma*}{Lemma}
\newtheorem*{fact*}{Fact}
\newtheorem{corollary}[subsubsection]{Corollary}
\newtheorem*{corollary*}{Corollary}
\newtheorem{conjecture}[subsubsection]{Conjecture}
\theoremstyle{definition}

\theoremstyle{remark}
\newtheorem{remark}[subsubsection]{Remark}
\newtheorem*{remarks}{Remarks}

\renewcommand{\comment}[1] {  }

\DeclareFontFamily{OT1}{rsfs}{}
\DeclareFontShape{OT1}{rsfs}{n}{it}{<-> rsfs10}{}
\DeclareMathAlphabet{\mathscr}{OT1}{rsfs}{n}{it}

\newcommand{\Ad}{\mathrm{Ad}}
\newcommand{\Res}{\mathrm{Res}}

\newcommand{\Z}{\mathbb{Z}}

\newcommand{\adele}{\mathbb{A}_k}
\newcommand{\adeleE}{\mathbb{A}_E}

\newcommand{\CC}{\mathbb{C}}

\newcommand{\PP}{\mathbb{P}}

\newcommand{\WW}{\mathbb{W}}

\renewcommand{\SS}{\mathbb{S}}

\newcommand{\Ind}{\operatorname{Ind}}
\newcommand{\Hom}{\operatorname{Hom}}

\newcommand{\Aut}{{\operatorname{Aut}}}
\newcommand{\Planch}{{\operatorname{Planch}}}
\newcommand{\Gm}{\mathbb{G}_m}

\newcommand{\GL}{\operatorname{GL}}

\newcommand{\PGL}{\operatorname{PGL}}

\newcommand{\SL}{\operatorname{SL}}
\newcommand{\Sp}{\operatorname{Sp}}
\newcommand{\Mp}{\operatorname{Mp}}

\newcommand{\SO}{{\operatorname{SO}}}

\newcommand{\Gal}{\operatorname{Gal}}

\newcommand{\tr}{\operatorname{tr}}

\newcommand{\diag}{{\operatorname{diag}}}
\newcommand{\ev}{\operatorname{ev}}

\newcommand{\disc}{{\operatorname{disc}}}
\newcommand{\temp}{{\operatorname{temp}}}

\newcommand{\Id}{\operatorname{Id}}

\newcommand{\oomega}{{\boldsymbol{\omega}}}

\newcommand{\LG}{{^LG}}

\begin{document}
\numberwithin{equation}{section}
\setcounter{tocdepth}{2}
\title[Howe duality and Euler factorization]{Plancherel decomposition of Howe duality and Euler factorization of automorphic functionals}
\author{Yiannis Sakellaridis}
\email{sakellar@rutgers.edu}
\address{Department of Mathematics and Computer Science, Rutgers University -- Newark, 101 Warren Street, Smith Hall 216, Newark, NJ 07102, USA. \medskip  \linebreak  and \medskip \linebreak
\hspace{1cm} Department of Mathematics,
School of Applied Mathematical and Physical Sciences,
National Technical University of Athens,
Heroon Polytechneiou 9,
Zografou 15780, Greece.}

\begin{abstract}
There are several global functionals on irreducible automorphic representations which are \emph{Eulerian}, that is: pure tensors of local functionals, when the representation is written as an Euler product $\pi = \otimes'_v \pi_v$ of local representations. The precise factorization of such functionals is of interest to number theorists and is -- naturally -- very often related to special values of $L$-functions.

The purpose of this paper is to develop in full generality the Plancherel formula for the Weil or oscillator representation, considered as a unitary representation of a reductive dual pair, and to use it in order to demonstrate a very general principle of Euler factorization: local factors are determined via the Langlands correspondence by a local Plancherel formula. This pattern has already been observed and conjectured in the author's prior work with Venkatesh in the case of \emph{period integrals}. Here, it is shown that the Rallis inner product formula amounts to the same principle in the setting of \emph{global Howe duality}.
\end{abstract}

\maketitle

\begin{flushright}
 \emph{To Roger Howe, \\ in admiration.}
\end{flushright}

\tableofcontents

\section{Introduction} 

The purpose of this paper is to develop the Plancherel formula for the Weil or oscillator representation, considered as a unitary representation of a reductive dual pair, and to show how known results on global Howe duality, more precisely Rallis' inner product formula (and its most recent extensions, in particular by Gan, Qiu and Takeda \cite{GQT}) can be reformulated as an identity between global and local functionals that generalizes a pattern already observed in the case of period integrals of automorphic forms.

Here is the general setup, which covers both period integrals and Howe duality (\emph{a.k.a}.\ theta correspondence, \cite{Howe}): One is given a reductive group $G$ over a global field $k$, and a unitary representation $\oomega$ of the adelic points $G(\adele)$, which is a restricted tensor product of local unitary representations $\oomega_v$. One is also given a suitable irreducible automorphic representation of $G$ with a factorization: \begin{equation}\label{pifactor}\pi= \bigotimes'_v \pi_v\end{equation} as a tensor product of irreducible representations of the local groups $G_v$, and an invariant pairing:
\begin{equation} \mathcal P: \pi \otimes \oomega^0 \to \CC\end{equation}
of ``global'' nature (i.e.\ given by some integral of automorphic functions over $[G]:=G(k)\backslash G(\adele)$), where $\oomega^0$ is a dense (Eulerian) subspace of smooth vectors in $\oomega$. For what follows we will ignore the difference between $\oomega$ and $\oomega^0$ in our notation, and formulate statements that hold for suitable dense subspaces of vectors. We denote by $\pi, \oomega$ etc.\ both the spaces of the representations and the corresponding actions of the group.

Examples of this setup include:
\begin{enumerate}
\item Period integrals over subgroups of $G$, which can equivalently be described as pairings:
$$ \pi\otimes \mathcal S(X(\adele))\to \CC,$$
where $X = H\backslash G$ is a homogeneous $G$-variety, and $\mathcal S$ denotes the space of Schwartz functions. Here $\oomega = L^2(X(\adele))$, and the pairing is given by:
\begin{equation}\mathcal P: \pi\otimes \mathcal S(X(\adele))\ni  \varphi \otimes \Phi \mapsto \int_{[G]} \varphi(g) \Sigma\Phi(g) dg,\end{equation}
where $ \Sigma \Phi$ denotes the automorphic function $\sum_{\gamma\in X(k)} \Phi(\gamma g)$, assuming that this integral converges.

\item The Weil representation $\oomega$, when $G=G_1 \times G_2$ denotes a dual pair, whereby for $\Phi\in \oomega^\infty$ and $\varphi_1\otimes\varphi_2$ in an automorphic representation $\pi=\pi_1\otimes\pi_2$, the pairing is given by the theta series of the pair $(\Phi,\varphi_1)$ integrated against $\varphi_2$, or vice versa. It can be written in the symmetric form:
\begin{eqnarray}\label{globaltheta}
 \mathcal P: \pi\otimes \oomega^\infty \ni \varphi_1\otimes \varphi_2\otimes \Phi \mapsto \nonumber \\ \int_{[G_1\times G_2]} \varphi_1(g_1) \varphi_2(g_2) \Sigma(\oomega(g_1,g_2)\Phi) d(g_1,g_2),
\end{eqnarray}
where $\Sigma$ denotes the standard automorphic functional on $\oomega^\infty$, provided these integrals converge.
\end{enumerate}

In general, not all integrals above are convergent, and one needs to suitably regularize them.

Suppose that, for some reason, one knows that the pairing $\mathcal P$ is Eulerian, that is: a pure tensor in the space
\begin{equation} \Hom_{G(\adele)}(\pi\otimes \oomega,  \CC) \simeq \bigotimes_v' \Hom_{G_v}(\pi_v\otimes \oomega_v,  \CC).\end{equation}
(This restricted tensor product also needs some qualifications, in general, depending on the subspace of vectors that one is considering.) This is automatically true in the multiplicity-one case, i.e.\ when the spaces 
$$\Hom_{G_v}(\pi_v\otimes \oomega_v,  \CC)$$ are (at most) one-dimensional, which is the case for Howe duality and for many spherical varieties. The question is, then, to describe an explicit factorization of $\mathcal P$ into local functionals. 

Such local factors will depend on how one fixes the isomorphism \eqref{pifactor}, hence there is a better hope of getting a meaningful answer if we double all variables, and consider the corresponding pairing $ \mathcal P\times\mathcal P$ on the pair $\pi\otimes\pi^\vee$, where $\pi^\vee$ is the dual representation of $\pi$, realized in a natural way on the space of automorphic representations. (Typically, $\pi$ will belong to $L^2([G])$ up to a character twist, and $\pi^\vee$ will be its complex conjugate up to the inverse twist.) The tensor product $\pi\otimes\pi^\vee$ has a canonical factorization:
\begin{equation}\label{pipifactor}\pi\otimes\pi^\vee= \bigotimes'_v (\pi_v\otimes\pi^\vee_v),\end{equation}
dictated by the preservation of the dual pairing, and so does the product of $\oomega$ with its dual (which, by unitarity, is just its complex conjugate). Therefore, we have a functional:
\begin{equation} \mathcal P\times \mathcal P: (\pi\otimes\pi^\vee) \otimes (\oomega\otimes\oomega^\vee)\to \CC,\end{equation}
which is invariant under $G(\adele)\times G(\adele)$ acting diagonally on the 1st and 3rd, resp.\ 2nd and 4th factors, and one is asking for an Euler factorization into local functionals with the analogous invariance property:
\begin{equation}\label{localfnls} (\pi_v\otimes\pi^\vee_v) \otimes (\oomega_v\otimes\oomega_v^\vee)\to \CC.\end{equation}

Equivalently, since the (admissible) dual of $\pi_v\otimes\pi^\vee_v$ is $\pi^\vee_v\otimes \pi_v$, we can view the global and local functionals as morphisms:
$$\oomega\otimes\oomega^\vee \to \pi^\vee\otimes\pi,$$
and because $\pi$ is assumed to be irreducible, these are determined uniquely by their composition with the canonical pairing
\begin{equation}\label{relchars}
J: \oomega\otimes\oomega^\vee \to \pi^\vee\otimes\pi \to \CC,\end{equation}
which we will call a \emph{relative character}. (It is a generalization of the notion of character, when $\oomega^0 = \mathcal S(H(\adele))$, the Schwartz space of a group $H$ under the $G=H\times H$-action.)

 If we denote by $J^\Aut$ (for ``automorphic'') the relative character that we get from the global pairing $\mathcal P$, we are seeking an Euler factorization:

\begin{equation}\label{EulerJ} J^\Aut = \prod_v J_v.\end{equation}

Such an Euler factorization was conjectured in \cite{SV} for the case of Eulerian period integrals, generalizing a conjecture of Ichino and Ikeda \cite{II} for the Gross-Prasad periods. The main idea is that the local relative characters $J_v$ of \eqref{EulerJ} will be provided by the \emph{Plancherel decomposition} of the unitary representations $\oomega_v$, assuming that this satisfies a \emph{relative local Langlands conjecture}. Namely, it is conjectured (and proven in several cases) that the Plancherel decomposition for $\oomega_v$ reads:
$$ \oomega_v = \int_{\widehat{G_{\oomega,v}}^\temp} \mathcal H_{\sigma} \mu_{\oomega,v}(\sigma),$$
where:
\begin{itemize}
\item the group $G_\oomega$ is a reductive group determined by the representation $\oomega$;
\item $\mu_{\oomega,v}$ is the standard Plancherel measure on the unitary dual of $G_{\oomega,v}=G_\oomega(k_v)$ (supported on the set of tempered representations);
\item The (possibly zero) Hilbert space $\mathcal H_{\sigma}$ is $\iota(\sigma)$-isotypic, where 
$$\iota: \widehat{G_{\oomega,v}}^\temp \to \widehat{G}_v$$ is a map from the tempered dual of $G_{\oomega,v}$ to the unitary dual of $G_v$ determined by a distinguished map of $L$-groups:
\begin{equation}\label{Lmap} {\LG_\oomega \times \SL_2 \to \LG}.\end{equation}
To be precise, the map of $L$-groups, including the ``Arthur'' $\SL_2$-factor  determines, by the Langlands and Arthur conjectures, a map $\phi\mapsto \Pi_{\iota(\phi)}$ from the set of tempered Langlands parameters into $\LG_\oomega$ to the set of (unitary) Arthur packets of $G_v$. Thus, in reality the above integral should be an integral over $L$-parameters into $\LG_\oomega$ (the actual group $G_\oomega$ plays no role), the measure $\mu_{\oomega,v}$ is the Plancherel measure on the set of those parameters (by conjectures of Hiraga-Ichino-Ikeda \cite{HII} it is well-defined, up to a small integer factor in the case of exceptional groups) and the space $\mathcal H_\sigma$ is isotypic for the set of representations in the Arthur packet $\Pi_{\iota(\phi)}$. However, for the purposes of the introduction we can ignore these fine differences.
\end{itemize}

Assuming this conjecture, for $\mu_{\oomega,v}$-almost every $\sigma$ we get a local relative character $J_{\sigma}^\Planch$ associated to the representation $\iota(\sigma)$, such that the Plancherel decomposition holds:
\begin{equation}\label{localPlancherel}
\left<\Phi_1, \Phi_2\right>_{\oomega_v} = \int_{\widehat{G_{\oomega,v}}^\temp} J_\sigma^\Planch(\Phi_1\otimes\bar\Phi_2) \mu_{\oomega,v}(\sigma).
\end{equation}
In other words, $J_\sigma^\Planch$ is pulled back from the inner product in the Hilbert space $\mathcal H_\sigma$.In practice, this relative character is continuous on the tempered dual of $G_{\oomega,v}$, and hence the $J_\sigma$'s are defined for every tempered $\sigma$.

Now assume that we are interested in the Euler factorization of the global relative character $J^\Aut$ for an automorphic representation $\pi$ which is a functorial lift $\iota(\sigma)$ of a tempered automorphic representation $\sigma$ of $G_\oomega$ via the above map of $L$-groups. (That is, $\pi$ belongs to the space of automorphic forms associated to the ``global Arthur parameter'' obtained by the ``global Langlands parameter'' of $\sigma$, if one can make sense of such global parameters, by composition with the above map.) Then the conjecture states, roughly, that up to a rational global factor which is missing from  \eqref{EulerJ}, the local factors $J_v$ are the ones given by the above Plancherel decomposition, that is:
\begin{equation} J_v = J_{\sigma_v}^\Planch.\end{equation}

The Euler product should be understood with the help of partial $L$-functions, as it will not converge, but almost all factors will be equal to a special value of an unramified $L$-function.

In the case of spherical periods, i.e.\ $\oomega=L^2(X(\adele))$ where $X$ is a (suitable) spherical variety, the $L$-group $\LG_\oomega$ is the $L$-group of that spherical variety, which can be defined in some generality -- s.\ \cite{SV} for split groups -- and is based on the dual group attached by Gaitsgory and Nadler \cite{GN}.

The purpose of this article is to unify the case of period integrals and the case of the theta correspondence, showing that the same principle of Euler factorization, as outlined above, is valid for the theta correspondence, as well. Here, the $L$-group and the map \eqref{Lmap} are provided by a conjecture of Adams \cite{Adams}. Essentially, if $G=G_1\times G_2$ is a dual pair, then $\LG_\oomega$ is the $L$-group of the ``smaller'' of the two, embedded diagonally in $\LG$; and the map from $\SL_2$ corresponds to a principal unipotent orbit in the commutator of its image. I give a more careful account of the appropriate conventions for $L$-groups for this case in \S \ref{ssLgroups}.

Of course, the Euler factorization of (the square of the absolute value of) the global Howe pairing \eqref{globaltheta} is known in most cases by the name of ``Rallis inner product formula'', and is a consequence of the Siegel-Weil formula developed   by the work of Rallis, Kudla, M{\oe}glin, Ichino, Jiang, Soudry, Gan, Qiu, Takeda and Yamana \cite{Rallis-oscillator, KR1,KR2,KR,Mo-theta,I1,I2,I3,JS,GT,GQT,Y1,Yamana}. In this article I restrict my attention to Howe duality for the so-called non-quaternionic, type $I$ dual pairs, i.e.\ symplectic-orthogonal or unitary. The goal of this article is to reinterpret the appropriate Rallis inner product formulas -- more precisely, those in the ``boundary'' and ``second term range'' in the sense of \cite{GQT} -- showing that the local factors of the Euler factorization of \eqref{globaltheta} are the ones of the local Plancherel formula \eqref{localPlancherel} for the Weil representation restricted to a dual pair.

To describe the contents of the paper in more detail, in section \ref{sec:Adams} I introduce the formalism of ``the $L$-group of Howe duality'' and Adams' conjecture  -- which are, however, only needed to establish the analogy with the ``relative Langlands program'' for spherical varieties, not for any of the results which follow. 

In section \ref{sec:local} I develop the Plancherel formula for Howe duality. It states (Theorem \ref{localtheorem}) that, if $G_1\times G_2$ is a dual pair (over a local field) with $G_2$ the ``small'' group, and if $\oomega$ denotes the oscillator representation of $G_1\times G_2$ (considered as a unitary representation) then $\oomega$ admits the Plancherel decomposition:
\begin{equation} \left<\Phi_1,\Phi_2\right> = \int_{\widehat{G_{2}}} J_\pi^\Planch (\Phi_1,\Phi_2)  \,\, \mu_{G_{2}}(\pi_2),
\end{equation}
where $\mu_{G_{2}}$ denotes Plancherel measure for $G_{2}$, and $\pi$ stands for the representation $\theta(\pi_2)\hat\otimes \pi_2$. The hermitian forms $J_\pi^\Planch$ are the ones explicitly defined by Jian-Shu Li in \cite{Li}, and  their positivity (proven under additional assumptions by Hongyu He in \cite{He}) is part of the theorem. The result is probably known to experts, and at least parts of it have appeared in the literature, cf.\ \cite{Howe-unitary, Gelbart, RS,OZ2,OZ1, GG}.

Finally, in section \ref{sec:global} I reformulate the Rallis inner product formula, in its most recent form as appearing in \cite{Yamana} and \cite{GQT}, in terms of the local hermitian forms $J_\pi^\Planch$ above. The main result (Theorem \ref{globaltheorem}) states that the relative characters \eqref{relchars} of global Howe duality admit a factorization as in \eqref{EulerJ}, whose local factors are precisely those hermitian forms.

It is by no means the first time that Howe duality and the theory of period integrals are being brought together; the theta correspondence has repeatedly been used to study period integrals, at least since the groundbreaking work of Waldspurger \cite{Waldspurger-Shimura,Waldspurger-torus}, and there has been a very systematically back-and-forth between these two methods in recent years around the program established by the Gan-Gross-Prasad conjectures \cite{GGP}. However, I am not aware of a uniform formulation of the principle of Euler factorization that was outlined above, and it may help in the understanding of some general principles underlying the theory of automorphic representations.

\subsection{Acknowledgements}

This paper is dedicated to Professor Roger Howe, in deep admiration of the  wealth of ideas that he has brought to the world of representation theory. 

The paper would not have been possible without the kind guidance of Wee Teck Gan who patiently answered my questions a few years ago and effectively guided me into the world of Howe duality. Moreover, the essential results that this paper is based on are all found in papers by him and his collaborators, which are in turn based on earlier work of Kudla, Rallis and others, building up on the groundbreaking ideas of Roger Howe.

I am also grateful to Jeff Adams for conversations on his conjecture on the Arthur parameters of Howe duality -- a conjecture which paved the way for unifying the theta correspondence with the Langlands program, based on the work of his and his collaborators in the real case.

Finally, I am indebted to the anonymous referee for a very prompt and comprehensive report, including several corrections and suggestions for improvement.

\section{Adams' conjecture and the $L$-group of Howe duality} \label{sec:Adams}

\subsection{Groups (overview)}

Fix a sign $\epsilon = \pm 1$, and a field $E$ of degree 1 or 2 over our base field $k$.

We will work in the context of Kudla-Rallis' Siegel-Weil formula, as generalized in \cite{GQT}. Thus, we will consider the theta correspondence between an (almost) arbitrary isometry group $G_1$ of an $\epsilon$-hermitian form on a vector space $V$, and the isometry group $G_2$ of an $(-\epsilon)$-hermitian form on a vector space $W$. When necessary, one of these groups will be replaced by its double metaplectic cover. Obviously, up to replacing $\epsilon$ by $-\epsilon$ (and except for a technical condition that we will eventually impose on the Witt rank of $V$ in order to apply known results on the Rallis inner product formula), the situation is symmetric in $V, W$ and we can interchange $G_1$ and $G_2$, but we will take $G_2$ to be the ``smaller'' of these groups.

More precisely, let $k$ be a number field, $E=k$ or a quadratic field extension, $\eta=\eta_{E/k}$ the (possibly trivial) quadratic idele class character of $k$ attached to $E$ by class field theory. The action of the Galois group of $E/k$ will be denoted by a bar ($\bar{~}$), with the understanding that it is trivial when $E=k$. We will sometimes use $F$ to denote a completion of $k$, in which case we will abuse notation (when no confusion arises) and denote again by $E$ the ring $E\otimes_k F$, and by $\eta_{E/F}$ the corresponding quadratic character of $F^\times$. The set of points of a variety $X$ over a completion $k_v$ will be denoted both by $X(k_v)$ and by $X_v$.

Fix a sign $\epsilon = \pm 1$, and consider a non-degenerate $\epsilon$-hermitian space $V$, that is: $V$ is a vector space over $E$ equipped with a non-degenerate $\epsilon$-hermitian form:
$$(v, w) = \epsilon\overline{(w,v)}.$$
Similarly, let $W$ denote a non-degenerate $(-\epsilon)$-Hermitian space of dimension $n$. 

We will denote by $G_1$ and $G_2$, respectively, certain central extensions of the isometry groups of $V$ and $W$ by  $\CC^1=\{z\in\CC^\times| \, |z| =1\}$. These covers split in most cases, so by abuse of notation we may also denote by $G_1$, $G_2$ the groups appearing in the following table:

\begin{equation}\label{table}\begin{array}{|r|c|c|c|c|c|}
\hline 
& G_1(V) & \check G_1 & \check G_2 & G_2(W) & d(n) \\
\hline (E:k)=2 & U_m & \GL_m & \GL_n & U_n & n \\
\epsilon = -1, \, n \mbox{ even} & \Sp_m & \SO_{m+1} & \SO_n & O_n & n-1 \\
\epsilon = -1, \, n \mbox{ odd} & \Mp_m & \Sp_m & \Sp_{n-1} & O_n & n-1 \\
\epsilon = 1, \, m \mbox{ even} & O_m & \SO_m & \SO_{n+1} & \Sp_n & n+1 \\
\epsilon = 1, \, m \mbox{ odd} & O_m & \Sp_{m-1} & \Sp_n & \Mp_n & n+1 \\
\hline
\end{array}\end{equation}

Here $\Mp_m$ denotes the double metaplectic cover of the symplectic group $\Sp_m$. (Our notation for symplectic groups uses $\Sp_m$ to denote the isometry group of a symplectic space of dimension $m$, so $m$ is even.)

The table above includes, besides the groups $G_1$ and $G_2$, their Langlands dual groups, i.e.\ the identity components of their $L$-groups. The value $d(n)$ is the value of $m$ corresponding to the \emph{boundary case}, that is: the case when the standard representations of $\LG_1$ and $\LG_2$ have the same dimension (if possible by parity restrictions). 

We will assume throughout (as we may without loss of generality, by symmetry) that $d(n)\le m$, i.e.\ $G_2$ (or, rather, its dual) is the smaller group, while $G_1$ is the larger one. 

The table above is provided for convenience of the reader -- however, there are many non-canonical choices that need to be made in order to identify the groups of Howe duality with the above groups, and their $L$-groups with the given $L$-groups. We are about to describe more canonical definitions, which will free us from the necessity to make such choices.

\subsection{Metaplectic group and the oscillator representation} \label{ssmetaplectic}

Fix an additive character $\psi: \adele/k\to \CC^\times$, and a factorization $\psi = \prod_v \psi_v$ into unitary characters of the completions $k_v$. Whenever no confusion arises, we will be using the same letter to denote the composition of $\psi$ with the trace map from $\adeleE$ to $\adele$, and for its restrictions to the various completions of $E$.

The space $V\otimes W$, considered as a vector space over $k$, has a natural symplectic structure. Restricting to a completion $k_v$, the associated Heisenberg group has a unique, up to isomorphism, irreducible representation $\oomega_v = \oomega_{v,\psi_v}$ where its center acts by $\psi_v$. This gives rise to a projective representation of $\Sp(V\otimes W)(k_v)$, and hence a representation of the group:
$$\GL(\oomega_v) \times_{\PGL(\oomega_v)} \Sp(V\otimes W)(k_v),$$
which is a central extension of $\Sp(V\otimes W)(k_v)$ by $\CC^\times$. This is the oscillator, or Weil, representation \cite{Weil}. Moreover, there is a canonical subextension by $\CC^1 = \{z\in \CC^\times| \, |z|=1\}$:
\begin{equation}\label{compactcovers} 1 \to \CC^1 \to \widetilde{\Sp}(V\otimes W)(k_v) \to \Sp(V\otimes W)(k_v) \to 1,
\end{equation}
characterized by the fact that it acts unitarily on $\oomega_v$. We will be working with $\CC^1$-extensions.

Let $G(V), G(W) \subset \Sp(V\otimes W)$ be the isometry groups of $V$, resp.\ $W$, and let $\tilde G_{1,v}$, $\tilde G_{2,v}$ be the preimages, in $\widetilde{\Sp}(V\otimes W)(k_v)$, of the images of $G(V)(k_v)$, resp.\ $G(W)(k_v)$ in $\Sp(V\otimes W)(k_v)$. These are central $\CC^1$-extensions of the corresponding classical groups. Note (as we will soon recall) that the isomorphism class of $\tilde G_{1,v}$ does not only depend on the space $V$, but also on $W$, and similarly for $\tilde G_{2,v}$.

Throughout this paper we will denote by $\oomega_v^\infty$ the smooth vectors of $\oomega_v$ with respect to the \emph{big} symplectic cover $\widetilde\Sp(V\otimes W)(k_v)$. The pull-back of $\oomega_v$ to $\tilde G_v:= \tilde G_{1,v} \times \tilde G_{2,v}$ gives rise to the theta correspondence (Howe duality). More precisely, for any irreducible representation $\pi$ of $\tilde G_{1,v}$ which occurs as a quotient of $\oomega_v^\infty$, there is a unique irreducible representation  $\theta(\pi)$ of $\tilde G_{2,v}$ such that $\pi\otimes \theta(\pi)$ occurs as a quotient, and vice versa when we interchange $G_{1,v}$ and $G_{2,v}$ \cite{Kudla-expo}. This was proven by Howe \cite{Howe} for Archimedean fields, by Waldspurger \cite{Waldspurger-Howe} for $p$-adic fields with $p\ne 2$, and by Gan and Takeda \cite{GT-Howe} for all $p$-adic fields.

The local theta correspondence is naturally a correspondence between \emph{genuine} representations of $\tilde G_{1,v}$, $\tilde G_{2,v}$, i.e.\ representations where the central $\CC^1$ acts by the identity character. 
 To translate this to a correspondence between representations of more classical groups, one needs to make some choices which give rise to splittings $G(V)_v\to \tilde G_{1,v}$, $G(W)_v\to \tilde G_{2,v}$, or at least (when this is not possible), splittings over the metaplectic double cover of $G(V)_v$, $G(W)_v$.

Such choices are described by Kudla in \cite{Kudla-splitting} and they are quite standard nowadays in the theory of Howe duality. There is a standard (described by Ranga Rao \cite{Rao}) set-theoretic splitting $\Sp(V\otimes W)(k_v)\to \widetilde{\Sp}(V\otimes W)(k_v)$ and an ensuing $2$-cocycle, valued in the group $\mu_8(\CC)$ of $8$-th roots of unity, which describes the $\CC^1$-extension $\widetilde{\Sp}(V\otimes W)(k_v)$. Given that, Kudla \cite[Theorem 3.1]{Kudla-splitting}, \cite[\S II.3]{Kudla-expo} describes explicit 1-cocycles $\beta_V: G(W)\to \CC^1$, $\beta_W: G(V)\to \CC^1$, depending on some choices, which trivialize the $2$-cocycle over $G(V), G(W)$, or over their double (metaplectic) covers. As the notation suggests, $\beta_V$ depends on $V$ and $\beta_W$ depends on $W$ (because the groups $\tilde G_{2,v}$, resp.\ $\tilde G_{1,v}$ do), and this dependence will also appear in the $L$-groups that we are about to define.

However, these choices are non-canonical and complicate the relationship between Howe duality and Langlands correspondence. Therefore, in this paper I will make the following convention: $G_{1,v}$, $G_{2,v}$ will, strictly speaking, denote the covering groups $\tilde G_{1,v}$, $\tilde G_{2,v}$ encountered above, and by ``representations'' of those groups we will always mean \emph{genuine} representations. In the next subsection, I will assign (non-standard) $L$-groups to these groups. I will also explain how certain choices give rise to usual $L$-groups, in parallel to Kudla's splitting of the covers. Following that, we will allow ourselves to abuse language and treat $G_{1,v}$, $G_{2,v}$ as a classical group or a double cover thereof, as in Table \ref{table}, when no confusion arises.

\subsection{$L$-groups} \label{ssLgroups}

We will adopt the following definitions of non-standard $L$-groups; notice that $L$-groups are defined globally, based on local considerations on covering groups. The definitions given here are compatible with the ones given by Adams \cite{Adams} in the Archimedean case, up to conjugacy (cf.\ Remark \ref{canonical-remark}). I will also explain how choices related to Kudla's cocycles correspond to choices that modify these $L$-groups into standard $L$-groups.

\begin{enumerate}
\item For $V$ a quadratic space of \emph{odd} dimension $m$ over $k$, we let 
\begin{equation}\label{Lqodd}\LG_1 = \Sp_{m-1}(\CC) \times \Gal(\bar k/k).\end{equation} Here the space $W$ is a symplectic space over $W$, and Kudla's cocycle $\beta_W$ giving rise to $G(V) \hookrightarrow \tilde G_{1,v}$ does not depend on any data of $W$. 

For $m$ odd we have: $O_m = \SO_m \times \Z/2$, and an irreducible representation for $O_m$ is given by an irreducible representation for $\SO_m$ and a sign for $\Z/2$. The $L$-group of $O_m$ should be identified with the $L$-group of $\SO_m$, with the sign of $\Z/2$ not affecting the $L$-parameter (or Arthur parameter) of a representation; this is compatible with results on the theta correspondence that will be recalled later (such as the results of Atobe-Gan \cite{AG}).


\item For $V$ a quadratic space of \emph{even} dimension $m$ over $k$, again the cocycle does not depend on any data of $W$, and we let: 
\begin{equation}\label{Lqeven}
\LG_1 =  O_m(\CC) \times_{\{\pm 1\}} \Gal(\bar k/k),\end{equation}
where $O_m$ maps to $\{\pm 1\}$ via the determinant and $\Gal(\bar k/k)$ maps via $\tilde\chi_V$, the (trivial or non-trivial) quadratic Galois character associated to the normalized discriminant of the quadratic space $V$. (``Normalized'' refers to the fact that for a split quadratic space it is a square, i.e.\ the normalized discriminant is the square class of $(-1)^{\frac{m}{2}}$ times the determinant of a matrix of the quadratic form.)

In the split case ($\tilde\chi_V =1$) this is just the direct product $\SO_m(\CC)\times\Gal(\bar k/k)$, but in the non-split case its quotient through the map $\Gal(\bar k/k)\to  \Gal(k(\sqrt{\disc(V)})/k)$ can be identified with $O_m(\CC)$. While this identification is customary, we should remind ourselves how it relates to the more ``standard'' version of the $L$-group of $\SO_m$ as a pinned semi-direct product:
\begin{equation}\label{LSO} \SO_m(\CC)\rtimes \Gal(k(\sqrt{\disc(V)})/k).\end{equation}  
We may fix a pinning for $\SO_m$. Any two such pinnings are conjugate by a unique element of $\SO_m(\CC)$, up to the center $\{\pm 1\}$ of the group. Identifying, then, the group $O_m(\CC)$ with the pinned semi-direct product $\SO_m(\CC)\rtimes \Gal(\bar k/k)$ depends on choosing an element $\epsilon \in O_m(\CC)\smallsetminus \SO_m(\CC)$ which acts by the outer automorphism on the pinning, and with $\epsilon^2 =1$. There are two inequivalent choices for such an $\epsilon$ (when $m>2$, whose quotient is the central $-1 \in \SO_m(\CC)$. \emph{We choose the $\epsilon$ which corresponds to a simple reflection} under the standard representation of $O_m(\CC)$, i.e.\ with eigenvalues $(1,1,\dots 1, 1, -1)$. This is the standard choice in the literature.

Finally, we mention that equivalence classes of Langlands or Arthur parameters for $O_m$ will be classes of parameters into the above $L$-group modulo the action of $O_m(\CC)$ by conjugation.

\item For $V$ a symplectic space of dimension $m$ (and hence $W$ a quadratic space, whose discriminant Galois character we will denote by $\tilde\chi_W$), the 1-cocycle $\beta_W$ depends on $W$, and only trivializes the cover if $\dim(W)$ is even; in the odd case, it reduces the cover $\tilde G_{1,v}$ to the double metaplectic cover $\Mp_m(k_v)$ of $\Sp_m(k_v)$. Hence, we distinguish two sub-cases:

\begin{itemize}
\item If $n=\dim(W)$ is even, we take:
\begin{equation}\label{Lseven}
\LG_1 =  O_{m+1}(\CC) \times_{\{\pm 1\}} \Gal(\bar k/k),\end{equation}
where $O_{m+1}$ maps to $\{\pm 1\}$ via the determinant and $\Gal(\bar k/k)$ maps via $\tilde\chi_W$.

Of course, this is still isomorphic to the direct product of $\SO_{m+1}(\CC)$ with $\Gal(\bar k/k)$, simply by multiplying by the diagonal of $\tilde\chi_W$. We view this operation at the level of $L$-groups as the analog of Kudla's cocycle, which identifies $\tilde G_{1,v}$ with $\CC^1 \times \Sp_m$. However, in view of the behavior of the theta correspondence in terms of Langlands parameters, it is better to adopt the above definition of $L$-group for $\tilde G_{1,v}$.

\item If $n = \dim(W)$ is odd, we take: 
\begin{equation}\label{Lsodd}
\LG_1 = \Sp_m(\CC) \times \Gal(\bar k/k),
\end{equation}
\emph{however}: this is \emph{not} the $L$-group of the double-cover metaplectic group $\Mp_m$ that is found in the literature. Namely, we recall from \cite[\S 4.3]{Weissman-split} that the choice of an additive character $\psi_v$ (and the standard choice of fourth root of unity $i\in \CC$) identifies the $L$-group of the metaplectic double cover $\Mp_m(k_v)$ with $\Sp_m(\CC)\times \Gal(\bar k/k)$. Kudla's cocycle (which depends on $W$) defines a splitting:
$$\Mp_m(k_v) \hookrightarrow \tilde G_{1,v}. $$
This splitting corresponds to an identification of $L$-groups: 
\begin{equation}\label{ident-symplectic}
\LG_1 = \Sp_m(\CC) \times \Gal(\bar k/k) \xrightarrow\sim {^L\Mp_m} = \Sp_m(\CC) \times \Gal(\bar k/k)
\end{equation}
\emph{given by multiplication by the quadratic character associated to $W$}, that is:
$$ (g,\sigma) \mapsto (\tilde\chi_W(\sigma) g, \sigma).$$

\end{itemize}

\item In the unitary case, Kudla's cocycle $\beta_W$ depends on the choice of a unitary character $\chi_W$ of (the local quadratic extension) $E^\times$ which extends the $n$-th power of the quadratic character $\eta$ associated to the extension $E/k_v$ (where, again, $n=\dim(W)$). Rather than making such a choice, we define the $L$-group $\LG_1$ as the inflation to $\Gal(\bar k/k)$ of the possibly non-split extension:
\begin{equation}\label{GLcover}
\widetilde{GL_m}^n := \left< \GL_m(\CC), \sigma | \sigma^2 = (-1)^n, \sigma g \sigma^{-1} = g^c \right>,
\end{equation}
where $\GL_m(\CC)$ is considered as a pinned group, and $g^c$ is the pinned Chevalley involution, in other words:
\begin{equation}
\LG_1 = \widetilde{GL_m}^n \times_{\Gal(E/k)} \Gal(\bar k/k),
\end{equation}
where, of course, $\widetilde{GL_m}^n$ maps to $\Gal(E/k)$ with $\sigma$ mapping to the non-trivial element.

The Langlands parameter of an idele class character $\chi_W$ of $E$ which extends the quadratic character $\eta_{E/k}^n$ is a homomorphism:
$$ \mathcal W_k \to {^L\Res_{E/k} \Gm} = (\CC^\times \times \CC^\times) \rtimes \Gal(E/k)$$
(where $\mathcal W_k$ denotes the Weil group of $k$) with the property that composing with the map:
$$ {^L\Res_{E/k} \Gm}  \to {^L\Gm} = \CC^\times \times \Gal(E/k)$$
given on connected components by $(z_1,z_2)\mapsto z_1 z_2$ we get the $n$-th power of the quadratic character of $\mathcal W_k$ associated to $E/k$. Such a parameter is necessarily of the form:
$$ w\mapsto (\tilde\chi_W ,\tilde\chi_W^\sigma),$$
where $\tilde\chi_W: \mathcal W_k \to \widetilde{GL_1}^n$ and $\tilde\chi_W^\sigma$ denotes its $\sigma$-conjugate in $\widetilde{GL_1}^n$. (I am confusing here maps to non-connected groups and their projections to the connected components, but the reader should have no difficulty discerning the meaning.) Thus, a choice of $\chi_W$ gives rise to an isomorphism:
\begin{equation}\label{LtoUnitary} \LG_1 \xrightarrow\sim \GL_m(\CC)\rtimes \Gal(\bar k/k)\end{equation}
(the latter being the $L$-group of the unitary group), namely \emph{multiplication by $\tilde\chi_W$}. We see this isomorphism as the identification of $L$-groups provided by the corresponding splitting:
$$ U_m(k_v) \hookrightarrow \tilde G_{1,v}$$
defined by Kudla. One readily checks from Kudla's formulas that multiplying a given $\chi_W$ by a character $\chi'$ of $U_1$ has the effect of multiplying the above splitting by $\chi'$ composed with the determinant, which is compatible with the modification of \eqref{LtoUnitary} that would ensue.
\end{enumerate}

This completes the description of $\LG_1$, and $\LG_2$ is defined in a completely analogous way, by interchanging the role of $V$ and $W$.

\begin{remark}
It would be interesting to see an account of how the covering groups $\tilde G_{1,v}$, $\tilde G_{2,v}$ might be seen as arising from the $K_2$-covers of Brylinski and Deligne \cite{BD}, and their $L$-groups above arising from the canonical $L$-groups attached to these covers by Weissman \cite{Weissman-Lgroups}.
\end{remark}

\subsection{The $L$-group and Arthur-$\SL_2$ of Howe duality}

From now on we denote by $G_{1,v}$, $G_{2,v}$ (or $G_1(k_v), G_2(k_v)$) the covering groups $\tilde G_{1,v}$, $\tilde G_{2,v}$, or their simplified versions of Table \eqref{table}, when no confusion arises. However, we \emph{do not} adopt the simplified versions of their $L$-groups appearing in that table, which depend on choices. Rather, we work throughout with the definitions of $\LG_1, \LG_2$ given in the previous subsection. We set $G_v = G_{1,v} \times G_{2,v}$, and $\LG = \LG_1 \times_{\Gal(\bar k/k)} \LG_2$.

From now on we assume throughout (as we may, by symmetry) that $d(n) \le m$, i.e.\ $G_2$ is the ``small'' group. We define \emph{the $L$-group of Howe duality} to be:
\begin{equation}
\LG_\oomega := \LG_2,
\end{equation}
endowed with a \emph{canonical morphism}:\footnote{S.\ Remark \ref{canonical-remark}.}
\begin{equation}\label{canonicalmorphism}
\LG_\oomega \times \SL_2 \to \LG,
\end{equation}
which is described as follows: To define the restriction of this map to $\LG_\oomega$, it is enough to define two $L$-morphisms (i.e.\ compatible with the quotient to $\Gal(\bar k/k)$: $\LG_\oomega\to\LG_1$ and $\LG_\oomega\to\LG_2$. The latter is taken to be the identity, and the former will be the ``natural'' morphism which will be described below. We then map $\SL_2$ to the centralizer of the image of $\LG_\oomega$ in the connected component of $\LG_1$.

Before we proceed to the description of the map 
\begin{equation}\label{Lembedding}\LG_\oomega \times \SL_2 = \LG_2 \times \SL_2 \xrightarrow{\gamma} \LG_1,\end{equation} let us assume it to recall Adams' conjecture, restricted to tempered representations of $G_2$. (Non-tempered representations of the small group will not play any role in this paper, and in any case the ``naive'' version of the conjecture  needs to be corrected in this case, see \cite[\S 7.1]{Moeglin-Adams}. 

\begin{conjecture}[Adams' conjecture]\label{Adamsconjecture}
The theta lift to $G_{1,v}$ of a tempered representation $\pi$ of $G_{2,v}$ with Langlands parameter $\phi$, if non-zero, belongs to an Arthur packet with Arthur parameter 
\begin{equation}
\phi': \mathcal W_{k_v}' \times \SL_2 \xrightarrow{\phi\times \Id} \LG_2\times \SL_2 \xrightarrow{\gamma} \LG_1,
\end{equation}
 where $\gamma$ is the morphism \eqref{Lembedding}. 
\end{conjecture}

Here $\mathcal W_{k_v}'$ denotes the Weil (in the Archimedean case) or Weil-Deligne group (in the non-Archimedean case) of $k_v$.

This conjecture, for tempered representations, has been proven in several cases over Archimedean fields by Adams, Barbasch, M{\oe}glin and Paul \cite{Adams, AB-complex, Moeglin-Archimedien, Paul1, Paul2, Paul3}. In the non-Archimedean case, it is very close to being a theorem, summarized in \cite[Theorems 4.3 and 4.5]{AG}, based on results of Atobe, Ichino, Gan and Savin \cite{Atobe, AG, GI-GP, GS}  extending previous results of Mui\'c \cite{Mu1,Mu2,Mu3,Mu4} and M{\oe}glin \cite{Moeglin-Adams, Moeglin-multiplicite}. More precisely, the relevant cases of the aforementioned theorems are when $m \ge m_1$ (statements (2),(3),(4)) in Theorem 4.3, and statements (1), (2), (3) in Theorem 4.5. In all cases but the ``base case'' of the ``going-up tower'' (i.e.\ statements (1), (2) in Theorem 4.5) it is immediate to see that the Langlands parameters given by \cite{AG} are the ones of the ``main'' Langlands packet inside our desired Arthur packet. In the remaining cases, one needs to argue that the stated representations belong to our desired Arthur packet, as was done in \cite[\S 8]{Moeglin-Adams} for symplectic-even orthogonal pairs.

We now come to a description of the morphism \eqref{Lembedding}, commenting on its relation with the definitions of Adams \cite{Adams} and the results of \cite{AG}. In all cases, after we describe the map from $\LG_\oomega$, we will map $\SL_2$ to the centralizer of the image of $\LG_\oomega$ in the identity component of $\LG_1$, so that the non-trivial unipotent orbit in $\SL_2$ maps into the \emph{principal} unipotent orbit in this centralizer. Hence, it remains to describe the map from $\LG_\oomega$.

\begin{enumerate}
\item When $V$ is a quadratic space of odd dimension $m$, we embed: 
\begin{equation}
\LG_\oomega = \Sp_n(\CC) \times \Gal(\bar k/k) \to \LG_1= \Sp_{m-1}(\CC) \times \Gal(\bar k/k)
\end{equation}
by identifying the symplectic space associated to the former with a symplectic subspace of the latter. 

Recall, however, from \eqref{ident-symplectic} that in this case the $L$-group $\LG_1$ is not the ``standard'' $L$-group of the two-fold cover of the metaplectic group, but a twist of it by the discriminant character of $V$. Therefore, the map \eqref{Lembedding} induces a map on $L$-parameters that corresponds to the one of \cite{AG}, where a twist by $\chi_V^{-1}$ appears. (The character $\chi_W$ of \emph{loc.cit}.\ is in this case trivial.)

\item When $V$ is a quadratic space of even dimension $m$, we embed:
\begin{equation}
\LG_\oomega = O_{n+1}(\CC) \times_{\{\pm 1\}} \Gal(\bar k/k) \to \LG_1= O_m(\CC) \times_{\{\pm 1\}} \Gal(\bar k/k)
\end{equation}
by identifying the quadratic space of the first group with a subspace of the quadratic space of the latter. Notice that for both $L$-groups, $\Gal(\bar k/k)$ maps to $\{\pm 1\}$ via the determinant of $V$, so the morphism is well-defined. 

To compare with the results of \cite{AG}, notice that our choices for the $L$-group of $G_2$ coincide, while to identify the $L$-groups of $\LG_1$ one needs to multiply by the quadratic character $\tilde\chi_V$ -- this accounts for the factor of $\chi_V^{-1}$ in their formulas, while $\chi_W$ is again trivial.

\item When $V$ is symplectic, $W$ is orthogonal and we repeat what was done in the previous two cases, with the only difference that the `small'' $L$-group is the one of the orthogonal space. Notice that now the character $\chi_V$ in the results of \cite{AG} will be trivial, while the character $\chi_W$ accounts for our twisted definition of $L$-groups. 

\item When $V$ is unitary and the difference $m-n$ of the dimensions of $V$ and $W$ is an even number $2s$, in which case both covers $\widetilde{\GL_m}^n$ and $\widetilde{\GL_n}^m$ of \eqref{GLcover} are simultaneously split or non-split, and having the ``standard'' pinning of general linear groups in mind, we embed $\GL_n(\CC)$ into the ``middle'' block of $\GL_m(\CC)$. The element $\sigma= \sigma_n$ of $\widetilde{\GL_n}^m$ will go to the corresponding element $\sigma=\sigma_m$ of $\widetilde{\GL_m}^n$ when $m$ and $n$ are even (the extensions are split), and to the element:
\begin{equation} \left(\begin{array}{ccc} I_s \\ &I_n \\ &&-I_s \end{array}\right) \sigma_m,\end{equation}
as in \cite{Adams}, in the odd case, in order to make this map a homomorphism. When $m-n$ is odd, so exactly one of the two covers $\widetilde{\GL_m}^n$ and $\widetilde{\GL_n}^m$ is non-split, we embed, as in \cite{Adams}, the connected component $\GL_n(\CC)$ into the top left block of $\GL_m(\CC)$, and map $\sigma_n$ to the element:
\begin{equation} \left(\begin{array}{cc} & I_n \\ I_{m-n} \end{array}\right) \sigma_m.\end{equation}

A more linear-algebraic description of Langlands parameters into $\LG_1$, $\LG_2$ will make the above definitions appear more natural, and will clarify the relation with the results of \cite{AG}: It is easy to see that equivalence classes (i.e.\ $\GL_m(\CC)$-conjugacy classes) of Langlands parameters into $\widetilde{\GL_m}^n$ are conjugacy classes of Frobenius-semisimple homomorphisms
$$ \phi:\mathcal W_E'\to \GL_m(\CC)$$
with the property that there is a non-degenerate bilinear form $B$ on $\CC^m$ with 
\begin{equation}\label{B1} B(\phi(w) x, \phi(\tau w\tau^{-1}) y) = B(x,y),\end{equation}
and
\begin{equation}\label{B2} B(y, x) = (-1)^{m+n+1} B(x, \phi(\tau^2) y),\end{equation}
where $\tau$ is some (any) chosen element of $\mathcal W_F \smallsetminus \mathcal W_E$.

It is now easy to see that Langlands parameters into $\LG_2$ naturally give rise to (equivalence classes of) Langlands parameters into $\LG_1$, and that this map between equivalence classes corresponds to the map of $L$-groups described above. Moreover, the description of parameters in \cite{AG} is identical, except that the factor $(-1)^{m+n+1}$ of \eqref{B2} is replaced by $(-1)^{m+1}$. This corresponds to the twist by $\tilde\chi_W$ that is needed to pass from $\LG_1$ to the $L$-group of $U_m$, as in \eqref{LtoUnitary}, and a similar twist by $\tilde\chi_V$ is needed to pass from $\LG_2$ to the $L$-group of $U_n$; this explains the factors $\chi_V^{-1}\chi_W$ appearing in the formulas of \cite{AG}.
\end{enumerate}

This completes the description of the ``$L$-group of Howe duality'', endowed with a ``canonical'' morphism \eqref{canonicalmorphism}.

\begin{remark}\label{canonical-remark}
The morphism \eqref{canonicalmorphism} that we defined \emph{ad-hoc} appears to be the ``correct'' one for equivalence classes of Langlands parameters; hence, it is the ``correct'' one up to conjugacy by the connected component $\check G$ of $\LG$. This makes it less canonical than the connected component $\check G_X$ of the $L$-group of a spherical variety $X$ which, in \cite{SV}, was given with a morphism:
$$ \check G_X \times \SL_2 \to \check G$$
canonical \emph{up to conjugacy by the canonical maximal torus} (and completely canonical if we were working with pinned groups). It would be desirable, not only for aesthetic reasons, to have a more geometric definition of $\LG_\oomega$, together with an analog of the ``boundary degenerations'' of a spherical variety for Howe duality, that would allow us to pinpoint a more canonical map \eqref{canonicalmorphism} of $L$-groups.
\end{remark}

\subsection{A Plancherel-theoretic version of Adams' conjecture}


Now consider the unitary oscillator representation $\oomega_v$ at a place $v$ as a genuine, unitary representation of the dual pair $G_v = G_{1,v}\times G_{2,v}$. Throughout this paper, the ``unitary dual'' of these groups means the \emph{genuine} unitary dual. 

The abstract theory of the Plancherel formula tells us that there is a decomposition:
\begin{equation}\label{Plabstract}
\oomega_v = \int_{\widehat{G_v}} \mathcal H_\pi \mu_v(\pi),
\end{equation}
where $\mu_v$ is a measure on the unitary dual, and for an irreducible unitary representation $\pi$ of $G_v$, the unitary representation $\mathcal H_\pi$ is isomorphic to a sum of copies of $\pi$.

For more careful presentations of the Plancherel decomposition, including issues of measurability, I point the reader to \cite{BePl} and \cite[\S 6.1]{SV}. For the decomposition \eqref{Plabstract} to make sense, one needs to specify morphisms from a dense subspace $\oomega_v^0$ of $\oomega_v$ to the Hilbert spaces $\mathcal H_\pi$, and the decomposition is essentially unique, in the sense that the resulting measure on $\widehat{G_v}$ which is valued in the space of Hermitian forms on $\oomega_v^0$ is unique.

We will later see that in this case one can take $\oomega_v^0=\oomega_v^\infty$, or, in the language of \cite{BePl}, the decomposition is \emph{pointwise defined} on $\oomega_v^\infty$. Assuming this, for ($\mu$-almost) every irreducible unitary $\pi = \pi_1\otimes\pi_2$ of $G_v$, the morphism $\oomega_v^\infty \to \mathcal H_\pi$ of the Plancherel decomposition has to factor through a (semisimple) $\pi_v$-isotypic quotient of $\oomega_v^\infty$. The Howe duality theorem implies that $\pi_1$ and $\pi_2$ completely determine each other, and that the quotient is multiplicity-free; in the notation used before, $\pi_1 = \theta(\pi_2)$ and $\pi_2=\theta(\pi_1)$. (Notice that we are throughout omitting the dependence on the character $\psi_v$ from the notation.) Hence, we have a Plancherel decomposition of the form:
\begin{equation}\label{Plabstract2}
\oomega_v = \int_{\widehat{G_{2,v}}} \theta(\pi_2)\hat\otimes \pi_2 \,\, \mu_{2,v}(\pi_2),
\end{equation}
where $\mu_{2,v}$ is now some measure on the unitary dual of $G_{2,v}$ (the push-forward of the measure $\mu_v$ on $\widehat{G_v}$). The fact that the decomposition is pointwise defined on $\oomega_v^\infty$ will be shown in the context of the proof of Theorem \ref{localtheorem}, in \S \ref{sspflocalthm}.




We formulate the following unitary variant of Adams' conjecture. It is the analog of the ``relative local Langlands conjecture'' \cite{SV}[Conjecture 16.2.2] for the $L^2$-space of a spherical variety. Recall that we are assuming that $G_2$ is the ``small'' group, i.e.\ $m\ge d(n)$.

\begin{conjecture}\label{unitaryAdams}
There is a direct integral decomposition:
\begin{equation} \oomega_v = \int_{[\phi]} \mathcal H_\phi  \mu_{2,v}(\phi),\end{equation}
where:
\begin{itemize}
\item $[\phi]$ runs over isomorphism classes of local tempered (i.e., bounded) Langlands parameters into $\LG_\oomega =\LG_2$;
\item $\mu_{2,v}$ is in the natural class of measures on the set of such Langlands parameters;
\item $\mathcal H_\phi$ is isomorphic to a (possibly empty) direct sum of irreducible representations belonging to the Arthur packet associated to the composition:
\begin{equation}\label{composition} \mathcal W_{k_v}' \times \SL_2 \xrightarrow{\phi \times \Id} \LG_\oomega \times \SL_2 \to {^L(G_1 \times G_2)},\end{equation}
where the last arrow is the canonical morphism \eqref{canonicalmorphism}.
\end{itemize}
\end{conjecture}

Notice that, by abuse of notation, we use the same symbol $\mu_{2,v}$ for the Plancherel measure on the unitary dual of $G_{2,v}$, and for a measure on the set of its tempered Langlands parameters. By ``class'' of the measure $\mu_{2,v}$, we mean, as in \cite{SV}, an equivalence class of measures that are absolutely continuous with respect to each other. Given a choice of additive character $\psi_v$, this equivalence class has a ``canonical'' representative which corresponds to the conjecture of Hiraga-Ichino-Ikeda \cite{HII} on formal degrees. For discrete parameters, this measure is:
\begin{equation} \mu_{2,v}(\{\phi\}) = \frac{1}{|\mathcal S_\phi^\sharp|}|\gamma(0, \phi, \Ad, \psi_v)|, \end{equation}
where $\mathcal S_\phi^\sharp$ is a finite group related to a centralizer of the parameter. Notice that the adjoint $\gamma$-factor appearing in the formula makes sense also for our non-standard versions of $L$-groups, and it remains unchanged under the isomorphisms with more ``classical'' $L$-groups discussed in \S \ref{ssLgroups}. 

It is clear that Conjecture \ref{unitaryAdams} follows immediately from Adams' conjecture \ref{Adamsconjecture}, once one knows that the Plancherel measure $\mu_{2,v}$ of the oscillator representation \eqref{Plabstract2} is absolutely continuous with respect to the Plancherel measure of the group $G_{2,v}$ -- we will prove this in Theorem \ref{localtheorem}, including the stated fact that the Plancherel decomposition is pointwise defined on $\oomega_v^\infty$. 

Given that, and choosing $\mu_{2,v}= \mu_{G_{2,v}}$, the Plancherel measure for $G_{2,v}$, in \eqref{Plabstract2} (determined by a choice of Haar measure on $G_{2,v}$, which will be done globally), and independently from Conjecture \ref{unitaryAdams}, we get canonical morphisms from \eqref{Plabstract2}:

\begin{equation} \oomega_v^\infty \to \theta(\pi_2)\hat\otimes\pi_2, \end{equation}
up to scalars of absolute value $1$, for $\mu_{2,v}$-almost every $\pi_2$. We will actually see that these morphisms are ``continuous'' in $\pi_2$, and hence well-defined (possibly zero) for every tempered $\pi_2$, and we will see that these morphisms factorize the square of the absolute value of the global theta pairing.

\section{Plancherel decomposition of the oscillator representation} \label{sec:local}

The discussion of $L$-groups and Conjecture \ref{unitaryAdams} in the previous section was formulated in such a way to establish the analogy with the theory of period integrals, but to proceed we do not need to invoke Langlands parameters.
 The following theorem implies Conjecture \ref{unitaryAdams} if one assumes Adams' conjecture \ref{Adamsconjecture}:

\begin{theorem}\label{localtheorem}
In the previous setting, there is a direct integral decomposition:
\begin{equation}\label{Plunitary} \oomega_v = \int_{\widehat{G_{2,v}}} \theta(\pi_2)\hat\otimes\pi_2  \,\, \mu_{G_{2,v}}(\pi_2),
\end{equation}
where $\mu_{G_{2,v}}$ denotes Plancherel measure for $G_{2,v}$ (depending on a choice of Haar measure on $G_{2,v}$), and it is understood that $\theta(\pi_2)$ can be zero.

The decomposition is pointwise defined on $\oomega_v^\infty$, and for almost all $\pi_2$ the hermitian form on $\oomega_v^\infty$ that is pulled back from the unitary structure of $\pi=\theta(\pi_2)\hat\otimes \pi_2$ is equal to:
\begin{equation}\label{JPlanch}
J_{\pi}^\Planch(\Phi_1\otimes\bar\Phi_2):= \sum_\varphi \int_{G_{2,v}/\CC^1} \left<\oomega_v(g) \Phi_1,\Phi_2\right>_{\oomega_v} \overline{\left<\pi_2(g) \varphi, \varphi\right>_{\pi_2}} dg,
\end{equation} 
where $\varphi$ runs over an orthonormal basis of $\pi_2$.
\end{theorem}

In other words, for $\Phi_1,\Phi_2 \in \oomega_v^\infty$ we have:
\begin{equation}
\left<\Phi_1,\Phi_2\right>_{\oomega_v} = \int_{\widehat{G_{2,v}}} J_\pi^\Planch(\Phi_1 \otimes \overline{\Phi_2}) \,\, \mu_{G_{2,v}}(\pi_2),
\end{equation}
where $\pi $ stands for $\theta(\pi_2)\hat\otimes\pi_2$, and the $J_\pi^\Planch$ are positive semi-definite hermitian forms (for $\mu_{G_{2,v}}$-almost all $\pi_2$).

In comparison to the abstract decomposition \eqref{Plabstract2}, this theorem specifies that the Plancherel measure for $\oomega_v$ is absolutely continuous with respect to the Plancherel measure on $G_{2,v}$, and determines the hermitian forms of the Plancherel decomposition. The result is probably known to experts, and at least parts of it have appeared in the literature -- s.\ \cite{Li} for the hermitian forms above and \cite{GG} for the determination of the measure in the ``stable range''; cf.\ also \cite{Howe-unitary, Gelbart, RS,OZ2,OZ1} for special cases.  
The proof of this theorem will be the goal of this section. For the rest of this section we omit the index $v$ from the notation, e.g.\ $\oomega=\oomega_v$, $G_2=G_{2,v}$, $\mu_{G_2} = \mu_{G_{2,v}}$ etc. At some points we may denote the fixed completion $k_v$ of our global field by $F$.

\begin{remark}
In the non-Archimedean case it has been proven by Yamana \cite[Lemma 8.6]{Yamana} (s.\ also \cite[Proposition 11.5]{GQT}) that the forms $J_\pi^\Planch$ are non-vanishing if and only if $\theta(\pi_2)\ne 0$. In particular, the support of Plancherel measure coincides with the (closure of the) set of tempered representations of $G_{2,v}$ which are distinguished by Howe duality. It is expected that this should also be true in the Archimedean case. (To emphasize again the analogy with spherical varieties, this is the analog of Theorem 6.4.1 in \cite{SV}.)

\end{remark}

\subsection{Growth of matrix coefficients of $\oomega$}

\begin{proposition}\label{inHCS}
The matrix coefficients of the oscillator representation, 
$$ g\mapsto  \left<\oomega(g) \Phi_1,\Phi_2\right>_\oomega,$$
for $\Phi_1, \Phi_2\in \oomega^\infty$, when restricted to $G_2$, lie in the Harish-Chandra Schwartz space $\mathscr C(G_2)$ of genuine functions on $G_2$.
\end{proposition}

This is essentially \cite[Corollary 3.4]{Li} -- actually, its proof is contained in the proof of Theorem 3.2 of this paper. However, we shall formulate its proof in terms of ``doubling zeta integrals'' (and a convergence result for those due to Gan and Ichino), in order to establish facts that we need for our global application. It should be pointed out that the interpretation of local zeta integrals in terms of matrix coefficients was observed already in \cite{Li-nonvanishing}.

For a recollection of the notion of Harish-Chandra Schwartz space $\mathscr C(G_2)$ for the points of reductive algebraic groups, and more generally groups of ``polynomial growth'', see \cite{BePl}. It coincides with the space of all smooth functions which are integrable against $\Xi(g)\left(\log\Vert g\Vert\right)^N$ for all $N$ (and similarly for their derivatives under the universal enveloping algebra, in the Archimedean case), where $\Xi$ is the Harish-Chandra $\Xi$ function and $\Vert \cdot \Vert$ is any norm of polynomial growth on the group. The same definitions apply to functions $\CC^1$-covers of algebraic groups, and we denote by $\mathscr C(G_2)$ the space of \emph{genuine} Harish-Chandra Schwartz functions on $G_2$, i.e.\ those on which the central $\CC^1$ acts by the identity character.

The proposition will follow from interpreting integrals against matrix coefficients of $\oomega$ as local zeta integrals of the doubling method for the standard $L$-function, and invoking well-known results for that case. While doing so, we will be careful about canonically fixing certain isomorphisms, for later use. In particular, we will also be making some comments about global measures etc.

Consider the $(-\epsilon)$-hermitian space $\WW = W \oplus (-W)$, where $(-W)$ denotes the same vector space with opposite hermitian form, and the direct sum is an orthogonal one. The diagonal copy $W^\diag\subset \WW$ is an isotropic subspace, and the oscillator representation $\tilde\oomega$ associated to the symplectic space $V\otimes\WW$ and the additive character $\psi$ has a Schr\"odinger model:
\begin{equation}\label{Schr-diagonal} \tilde\oomega \simeq L^2(V\otimes \WW/\ell),\end{equation}
where $\ell$ is the Lagrangian $V\otimes W^\diag$. In the next subsection, we will recall that there is a canonical model up to canonical isomorphism for the oscillator representation, a canonical Haar measure (given the additive character $\psi$) on $V\otimes \WW/\ell$ and a distinguished isomorphism between the canonical model for $\tilde\oomega$ and $L^2(V\otimes \WW/\ell)$; hence, we can consider \eqref{Schr-diagonal} as a canonical isomorphism.

By the symplectic form, we can identify the quotient $V\otimes \WW/\ell$ with the $F$-linear dual $\ell^*$.
The space of smooth vectors is $\tilde\oomega^\infty = \mathcal S(\ell^*)$, the space of Schwartz functions. 
 The functional:
\begin{equation}\ev_0:\mathcal S(\ell^*) \ni \Phi\mapsto \Phi(0)\end{equation}
is an eigen-functional for the Siegel parabolic $\SS \subset \widetilde{\Sp}(V\otimes \WW)$ stabilizing $\ell$, with eigencharacter which we temporarily denote by $\chi$. Thus, by Frobenius reciprocity the functional defines a morphism:
\begin{equation}\label{toinduced}
\tilde\oomega^\infty \to \Ind_{\SS}^{\widetilde{\Sp}(V\otimes \WW)}(\chi),
\end{equation}
where $\Ind$ denotes \emph{unnormalized} induction.

The cover $\widetilde\Sp(V\otimes (-W))$ can canonically be identified with $\widetilde{\Sp}(V\otimes W)$ through an \emph{anti-genuine} involution (i.e.\ the central $\CC^1$ maps to $\CC^1$ through the inverse character). Thus, the oscillator representation associated to the symplectic space $V\otimes (-W)$ and the character $\psi$ is canonically identified with the dual $\oomega^\vee = \overline{\oomega} = \oomega_{\psi^{-1}}$. 

The embedding $\Sp(V\otimes W) \times \Sp(V\otimes W) \to \Sp(V\times \WW)$ gives rise to a morphism:
\begin{equation}
\widetilde{\Sp}(V\otimes W) \times \widetilde{\Sp}(V\otimes W) \to \widetilde{\Sp}(V\otimes \WW)
\end{equation}
which is genuine in the first copy and anti-genuine in the second. Restricting the oscillator representation of the big group, there is an isomorphism:
\begin{equation}\label{double} \oomega \hat \otimes \overline\oomega \xrightarrow{\sim} \tilde\oomega|_{\widetilde{\Sp}(V\times W) \otimes \widetilde{\Sp}(V\otimes W)}.\end{equation}
As we will recall in the next subsection, there is a canonical choice of a unitary such isomorphism, and it has the following property: For any maximal isotropic subspace $Y\subset V\otimes W$ whose linear dual $Y^*$ is endowed with a Haar measure, there are \emph{canonical} unitary isomorphisms with Schr\"odinger models:
\begin{equation}\label{can1}
\oomega \hat \otimes \overline\oomega \simeq L^2(Y^*)\hat\otimes \overline{L^2(Y^*)} 
\end{equation}
and
\begin{equation}\label{can2}
\tilde\oomega \simeq L^2(Y^* \oplus Y^*),
\end{equation}
and then \eqref{double} is the canonical isomorphism:\footnote{Throughout, a bar over a vector space denotes the same space with the conjugate $\CC$-action, and of course for $L^2$-spaces we have a linear isomorphism $L^2\to\overline{L^2}$ given by $\Phi \mapsto \bar\Phi$.}
\begin{eqnarray}\label{isomY} L^2(Y^*)\hat\otimes \overline{L^2(Y^*)} &=& L^2(Y^* \oplus Y^*), \\
\nonumber \Phi_1 \otimes \overline{\Phi_2} &\mapsto & \Phi(x,y) := \Phi_1(x) \Phi_2(y).
\end{eqnarray}
(And, similarly, subspaces of smooth vectors are identified with the corresponding spaces of Schwartz functions.)

Moreover:
\begin{lemma}\label{ev0}
For the canonical choices of isomorphisms \eqref{Schr-diagonal} and \eqref{double}, 
the composition of $\ev_0$ with the morphism \eqref{double} is the functional:
\begin{equation} \Phi_1\otimes\overline{\Phi_2} \mapsto \left<\Phi_1,\Phi_2\right>_\oomega.\end{equation}
\end{lemma}

For our current purposes an equality up to a scalar of absolute value one would be just as good; however, for our global application the exact equality is important, combined with the compatibility of the canonical isomorphisms with global models. This lemma is well-known, cf.\ \cite[p.182]{Li-nonvanishing}, but we will repeat its proof in the next subsection to clarify the canonical nature of various isomorphisms.

Assuming this lemma from now, we conclude:

\begin{corollary}
The matrix coefficients $g\mapsto \left<g \Phi_1,\Phi_2\right>_\oomega$,  lie in the space of the induced representation $\Ind_{\SS}^{\widetilde{\Sp}(V\otimes \WW)}(\chi)$, identified with the space of functions $\{f: \widetilde{\Sp}(V\otimes \WW) \to \CC | f(pg) =\chi(p) f(g)\,\, \forall p\in \SS, g\in \widetilde{\Sp}(V\otimes \WW)\} $, and restricted as functions to $\widetilde{\Sp}(V\otimes W) \times 1 \subset  \widetilde{\Sp}(V\otimes \WW)$. 
\end{corollary}

Let now $G(\WW)$ denote the preimage of the isometry group of $\WW$ (considered as a subgroup of $\Sp(V\otimes \WW)$) in $\widetilde{\Sp}(V\otimes \WW)$, and $\PP$ the corresponding Siegel parabolic of $G(\WW)$ fixing the isotropic space $W^\diag$. The restriction of an element of $\Ind_{\SS}^{\widetilde{\Sp}(V\otimes \WW)}(\chi)$, again considered as a function, to $G(\WW)$ via the natural embedding $G(\WW)\to \widetilde{\Sp}(V\otimes\WW)$ belongs to the space of the induced representation $\Ind_{\PP}^{G(\WW)}(\chi|_{\PP})$. 
Using \emph{normalized} (unitary) induction now, which we will denote by $I_{\PP}^{G(\WW)}$, it follows that the matrix coefficients of $\oomega$, restricted to $G_2 \subset \widetilde{\Sp}(V\otimes W)$ are restrictions via $G_2\to G_2\times 1 \hookrightarrow G(\WW)$ of the normalized induced representation 
$$ I_{\PP}^{G(\WW)} (\chi|_{\PP}\cdot \delta_{\PP}^{-\frac{1}{2}})$$
of $G(\WW)$, where $\delta_{\PP}$ denotes the modular character of $\PP$. 

We will now determine the growth of the character $\chi\delta_{\PP}^{-\frac{1}{2}}$. It is more convenient here to work with classical groups, so consider the quotient $\PP\to \Res_{E/F}\GL_n$, where $\Res_{E/F}\GL_n$ is identified with the Levi quotient of the Siegel parabolic of the classical group $\Sp(\WW)$, in such a way that the canonical central cocharacter composed with the (left) adjoint action on the unipotent radical is positive. Notice that the absolute value of the genuine character $\chi$ factors through a character of $\Res_{E/F}\GL_n$ -- it is this character that the following lemma is referring to.

\begin{lemma}\label{chargrowth}
The absolute value of the character $\chi\delta_{\PP}^{-\frac{1}{2}}$, considered as a character of $\Res_{E/F}\GL_n$ as above, is:
\begin{equation}
|\chi\delta_{\PP}^{-\frac{1}{2}}| = |\det|^{s_0},
\end{equation}
where $|\cdot |$ denotes the absolute value of the field $E$, and
\begin{equation}\label{s0}
s_0= \frac{m-d(n)}{2} .
\end{equation}
\end{lemma}

Notice that $s_0\ge 0$ according to our convention that $m\ge d(n)$.

\begin{proof}
This is well-known, and a fundamental fact for the theory of the Siegel-Weil formula, but for convenience of the reader we verify it again: by the explicit description of the Schr\"odinger model \cite[Theorem 3.5]{Rao}, the Siegel Levi $\SS$ of $\widetilde{\Sp}(V\otimes \WW)$ acts on the functional $\ev_0$ by a character $\chi$ with $|\chi|= |\det|^\frac{1}{2}$. The map $\PP\to \SS$ corresponds to tensoring the standard representation $\WW$ of the Levi of $\PP$ by $V$ (and changing scalars from $E$ to $F$), thus the determinant of the image is the $m$-th power of the determinant of (the Levi of) $\PP$. This explains the factor $|\det|^{\frac{m}{2}}$, while the factor $|\det|^{-\frac{d(n)}{2}}$ arises by calculation of the modular character of $\PP$. 
\end{proof}

Thus, to complete the proof of Proposition \ref{inHCS}, up to the proof of Lemma \ref{ev0}, it suffices to show that for any $f\in I_{\PP}^{G(\WW)} (\chi\delta_{\PP}^{-\frac{1}{2}})$, the function:
$$ G_2\ni g\mapsto f((g,1)),$$
where $g$ is embedded as above, lies in the Harish-Chandra Schwartz space. This follows from the local theory of the doubling method, more precisely from the following (or, rather, its proof):

\begin{lemma}
For any tempered genuine representation $\pi$ of $G_2$, any two smooth vectors $v_1 \in \pi$, $v_2\in \pi^*$, and any smooth vector $f\in I_{\PP}^{G(\WW)} (\chi')$, the local zeta integral:
\begin{equation}\label{localzeta} Z_F(v_1, v_2, f) := \int_{G_2/\CC^1} f((g,1)) \overline{\left< \pi(g) v_1, v_2\right>} dg\end{equation}
is absolutely convergent  if $|\chi'| = |\det|^s$ with $s>-\frac{1}{2}$.
\end{lemma}

\begin{proof}
See \cite{GI-formal}[Lemma 9.5].
\end{proof}

The proof of this lemma in \emph{loc.cit}.\ actually shows that the function $g\mapsto f((g,1))$ is in the Harish-Chandra space $\mathscr C(G_2)$. By Lemma \ref{chargrowth}, and since we are assuming that $m\ge d(n)$, this applies to the induced representation $I_{\PP}^{G(\WW)} (\chi\delta_{\PP}^{-\frac{1}{2}})$, completing the proof of Proposition \ref{inHCS}, except for Lemma \ref{ev0} which will be proven in the next section.

\subsection{Comparison of Schr\"odinger models}\label{sscomparison}

Suppose that $Z$ is a symplectic space (over our fixed local field which we keep omitting from the notation), $\oomega$ the oscillator representation of $\widetilde\Sp(Z)$ determined by the additive character $\psi$. If $\ell_1,\ell_2$ are two Lagrangian subspaces, endowed with Haar measures $dx_1, dx_2$, then $\oomega$ can be realized on the spaces $L^2(\ell_1^*, dx_1)$ and $L^2(\ell_2^*, dx_2)$, and there is a unique up to a scalar of absolute value $1$ equivariant isometry:
\begin{equation}
\mathcal F_{\ell_1,\ell_2}: L^2(\ell_1^*, dx_1)\xrightarrow\sim L^2(\ell_2^*, dx_2).
\end{equation}

Following \cite{Thomas}, we will now describe more canonical Schr\"odinger models, replacing \emph{functions} by \emph{half-densities} valued appropriate \emph{line bundles}, and equipped with canonical isometries between them  (see also \cite{WWLi-masters} for a nice presentation, and \cite{Lion-Vergne} for an earlier account of canonical intertwining operators). 

Namely, for a Lagrangian subspace $\ell$, let $\mathcal H_\ell$ denote the space of $L^2$-half densities\footnote{I use ``densities'' for the sheaf of densities on our vector spaces, not just for the ``Haar'' densities of \emph{loc.cit}. Thus, our densities are functions on the vector spaces times the ``Haar'' densities of \emph{loc.cit} -- when they are valued in the trivial line bundle their squares are (arbitrary) measures on the vector spaces, and when they are valued in a hermitian line bundle the same is true for the squares of their absolute values.} on $Z/\ell$ valued in the hermitian line bundle whose sections are functions on $Z$ satisfying: \begin{equation}\label{conddensities} f(\lambda+ z) = \psi\left(\frac{\left<z,\lambda\right>}{2}\right) f(z)
\end{equation}
 for all $\lambda\in \ell$, $z\in Z$, where $\left< \, \, , \, \, \right>$ denotes the symplectic pairing. We will see below that there are canonical intertwiners between the spaces $\mathcal H_\ell$ for different choices of $\ell$, so we will henceforth consider them as \emph{canonical} models for the oscillator representation.

A choice of splitting of the quotient $Z\to Z/\ell$ trivializes this line bundle, and a choice of Haar measure $dx$ on $Z/\ell= \ell^*$ (identified through the symplectic pairing, $z\mapsto \left< z, \cdot\right>$) whose positive square root is a half-density, turns half-densities into functions, thus giving rise to a unitary isomorphism:
\begin{equation}\label{HtoL} \mathcal H_\ell \xrightarrow\sim L^2(\ell^*, dx).\end{equation}

Later, when we will return to this point for global purposes, there are canonical Haar measures on the adelic points of these Lagrangians, and any splitting of $Z\to Z/\ell$ over the global field $k$ will result in the same trivialization of the aforementioned line bundle over the adelic points of $Z/\ell$ (because $\psi$ is automorphic), hence we can apply the present discussion identifying functions with half-densities in a canonical way.

Now, for two such Lagrangians $\ell_1$, $\ell_2$, there is a canonical (given the character $\psi$) self-dual Haar measure $\mu_{\ell_1,\ell_2}$ on the symplectic space $(\ell_1 + \ell_2)/(\ell_1\cap\ell_2)$. Using the canonical isomorphism:
\begin{equation}
(\ell_1 + \ell_2)/(\ell_1\cap\ell_2) \simeq \ell_1/(\ell_1\cap\ell_2) \oplus \ell_2/(\ell_1\cap\ell_2),
\end{equation}
the square root $\mu_{\ell_1,\ell_2}^\frac{1}{2}$ is a half density on the space on the right. The formula:

\begin{equation} \label{intertwiner} \mathcal F_{\ell_1,\ell_2} (\phi) (y) = \int_{x\in \ell_2/\ell_1\cap\ell_2} \phi(x+y) \psi\left(\frac{\left<x,y\right>}{2}\right) \mu_{\ell_1,\ell_2}^\frac{1}{2}
\end{equation}
defines the canonical intertwiner: $\mathcal H_{\ell_1}\to \mathcal H_{\ell_2}$.

Now we consider special cases that were encountered in the previous subsection.

First of all, there is a canonical isomorphism:
$$ \oomega\hat\otimes\bar\oomega \xrightarrow\sim L^2(\ell^*,dx) \hat\otimes \overline{L^2(\ell^*, dx)}$$
for any Lagrangian $\ell$ whose dual is endowed with a Haar measure $dx$, determined by the condition that the unitary pairing between $\oomega$ and $\bar\oomega$ translates to the inner product on $L^2(\ell^*, dx)$. This is \eqref{can1}.

Secondly, if $Z$ is a symplectic space and $\Z = Z\oplus (-Z)$, and $\ell$ is the diagonal of $Z$, there is a canonical splitting $\Z/\ell \to \Z$ whose image is the first copy of $Z$, and it is endowed with the self-dual Haar measure with respect to the character $\psi$. This gives a canonical isomorphism of the oscillator representation for $\Z$ with $L^2(Z)$, which is \eqref{Schr-diagonal}.

In the same setting, if $Y$ is any Lagrangian of $Z$ whose dual is endowed with a Haar measure $dx$, then \emph{any} choice of isotropic complement $X$ to $Y$ in $Z$ (giving rise to a complement $X\oplus X$ to $Y\oplus Y$ in $\Z$) induces \emph{the same} isomorphism between the oscillator representation attached to $\Z$ and the space:
$$ L^2(Y^*\oplus Y^*, dx \times dx).$$
Indeed, a different choice of splitting for the quotient $Z\to Y^*=Z/Y$ differs by a linear map $\kappa: Y^*\to Y$. Hence, for a vector $v\in \mathcal H_{Y\oplus Y}$, its images $\Phi, \Phi'$ under the two resulting isomorphisms:
$$ \mathcal H_{Y\oplus Y} \xrightarrow\sim L^2(Y^*\oplus Y^*, dx \times dx)$$
satisfy:
$$ \Phi' (z) = \psi\left(\frac{\left<z,\kappa(z)\right>}{2} - \frac{\left<z,\kappa(z)\right>}{2} \right) \Phi(z) = \Phi(z),$$
the minus sign in the bracket appearing because the symplectic forms on $Z$ and $(-Z)$ are opposite. This establishes \eqref{can2}.

Now we specialize to the setting of Lemma \ref{ev0}, except that we denote our symplectic spaces by $Z, \Z$ instead of $V\otimes W$, $V\otimes\WW$. The oscillator representation associated to $Z$ will be denoted by $\oomega$, and that associated to $\Z$ by $\tilde\oomega$. Take any Lagrangian $Y\subset Z$, and choose a Haar measure on $Y^*$. Set $\ell_1 = Y\oplus Y $, $\ell_2 = Z^\diag $. 

Choose a Lagrangian complement $X$ of $Y$ (identified with $Y^*$ via the symplectic pairing), and a Haar measure $dx$. As discussed before, this identifies $\tilde\oomega \simeq L^2(X\oplus X, dx \times dx)$. Moreover, we also saw that the first copy of $Z$ in $\Z$ identifies $\tilde\oomega$ with $L^2(Z, dz)$, where $dz$ is the self-dual measure with respect to $\psi$. For $\phi_1 \in \mathcal H_{\ell_1}$ and $\phi_2 =\mathcal F_{\ell_1,\ell_2} \phi_1 \in \mathcal H_{\ell_2}$, the corresponding functions $\Phi_1 \in L^2(X\oplus X, dx \times dx)$, $\Phi_2\in L^2(Z,dz)$ are obtained by using Haar measures to turn half-densities into functions, and restricting to $X\oplus X$, resp.\ $Z\oplus 0$. 
More precisely, splitting the self-dual Haar measure $dz$ on $Z=X\oplus Y$ as $dx \times dy$, identifying the quotient $\ell_2/\ell_1\cap\ell_2$ with $X^\diag$, and applying formula \eqref{intertwiner} divided by the half-density $dz^{\frac{1}{2}}$ we get:
$$ \Phi_2(x_0+y_0) = \frac{\phi_2}{dz^{\frac{1}{2}}} ((x_0+y_0,0)) \xlongequal{\eqref{intertwiner}} $$
$$ \int_{x\in X} \phi_1((x+x_0+y_0, x)) \psi\left(\frac{\left<(x,x),(x_0+y_0,0)\right>}{2}\right) \frac{\mu_{\ell_1,\ell_2}^\frac{1}{2}}{dx_0^\frac{1}{2}dy_0^\frac{1}{2}}  \xlongequal{\eqref{conddensities}}$$
$$\int_{x\in X} \phi_1((x+x_0, x)) \psi\left( \frac{\left<(x+x_0,x),(y_0,0)\right>}{2}\right) \psi\left(\frac{\left<(x,x),(x_0+y_0,0)\right>}{2}\right) \frac{\mu_{\ell_1,\ell_2}^\frac{1}{2}}{dx_0^\frac{1}{2}dy_0^\frac{1}{2}} $$
$$ = \psi\left(\frac{\left<x_0,y_0\right>}{2}\right) \int_{x\in X} \phi_1((x+x_0, x)) \psi\left(\left<x,y_0\right>\right) \frac{\mu_{\ell_1,\ell_2}^\frac{1}{2}}{dx_0^\frac{1}{2}dy_0^\frac{1}{2}}.$$

Now, we can write $\mu_{\ell_1,\ell_2}^\frac{1}{2} = dx^\frac{1}{2} \times dy_0^\frac{1}{2}$, where we have identified $\ell_2/\ell_1\cap\ell_2$ with $X^\diag$ as before, and $\ell_1/\ell_1\cap \ell_2$ with $Y\oplus 0$. Moreover, $\phi_1((x+x_0,x)) = \Phi_1((x+x_0,x)) dx^\frac{1}{2} dx_0^\frac{1}{2}$. The last formula reads, then:
\begin{equation}
\Phi_2(x_0+y_0)  = 
\psi\left(\frac{\left<x_0,y_0\right>}{2}\right) \int_{x\in X} \Phi_1((x+x_0, x)) \psi\left(\left<x,y_0\right>\right) dx.
\end{equation}

In particular, for $x_0=y_0 = 0$ we obtain:
$$ \Phi_2(0) = \int_{x\in X} \Phi_1((x, x)) dx,$$
and if $\Phi_1$ is in the image of a vector of $\oomega\hat\otimes\bar\oomega$ as in \eqref{isomY}, we obtain the statement of Lemma \ref{ev0}.

\subsection{Plancherel decomposition of the oscillator representation}\label{sspflocalthm}

We are now ready to prove Theorem \ref{localtheorem}.

We have seen in Proposition \ref{inHCS} that the matrix coefficients of $\oomega^\infty$ lie in the Harish-Chandra Schwartz space of $G_2$. The theorem will now follow from the pointwise Plancherel decomposition for the Harish-Chandra Schwartz space of the group $G_2$, provided that we can show that the hermitian forms $J_\pi^\Planch$ of \eqref{JPlanch} are positive semi-definite. This has been proven by Hongyu He \cite{He} under some additional assumptions. Notice that the fact that the Plancherel measure for $\Omega$ is supported on the tempered dual of $G_2$, given that the diagonal matrix coefficients of a dense subspace lie in $L^{2+\epsilon}(G)$, is a theorem of Cowling, Haagerup and Howe \cite[Theorem 1]{CHH}. We will rely on and extend that theorem, using the Plancherel formula for the Harish-Chandra Schwartz space:

\begin{proposition} 
Let $\Omega$ be a unitary representation of (the points of) a reductive group $G$ over a local field, and $\Omega^0$ a dense, invariant subspace of (necessarily smooth) vectors whose matrix coefficients lie in the Harish-Chandra Schwartz space $\mathscr C(G)$. 

For all tempered representations $\pi$ consider the hermitian forms, depending on the choice of a Haar measure on $G$:
\begin{equation}\label{defJpi}
J_\pi(\Phi_1,\Phi_2) = \sum_v \int_{G} \left<\Omega(g) \Phi_1,\Phi_2\right>_\Omega \overline{\left<\pi(g) v, v\right>_\pi} dg,
\end{equation}
the sum running over an orthonormal basis of $\pi$.

Then $J_\pi$ is positive semi-definite for every tempered representation $\pi$ of $G$, and $\Omega$ admits the following Plancherel formula, pointwise defined on $\Omega^0$:
\begin{equation}\label{PlOmega} \left<\Phi_1,\Phi_2\right>_{\Omega} = \int_{\widehat{G}} J_\pi(\Phi_1,\Phi_2)  \mu_{G}(\pi),
\end{equation}
where $\mu_G$ denotes the Plancherel measure of $G$ corresponding to the choice of Haar measure in \eqref{defJpi}. 

\end{proposition}

The proposition immediately extends to genuine functions on $G_2$, and the notion of Harish-Chandra space for those discussed after the statement of Proposition \ref{inHCS}. This is where we will apply it, but for notational simplicity we present the proof for classical groups.

\begin{proof}

The formula \eqref{PlOmega} follows from the pointwise Plancherel decomposition for the Harish-Chandra Schwartz space of the group $G$ \cite{HC,WaPl}: Any Harish-Chandra Schwartz function $F$ on $G$ has the pointwise decomposition:
\begin{equation}
F(g) = \int_{\widehat{G}} F_\pi(g) \mu_G(\pi)
\end{equation}
with 
\begin{equation}F_\pi(x) = \tr(\bar\pi(\rho(x) F dg)) =  \sum_v \int_{G(F)} F(gx) \overline{\left<\pi(g) v, v\right>_\pi} dg,\end{equation}
where $\rho$ denotes the right regular representation. In particular, when 
$F$ is the matrix coefficient $F(g) = \left<\Omega(g) \Phi_1,\Phi_2\right>$, 
one has
$$F_\pi(1) = J_\pi (\Phi_1 , \Phi_2).$$

 We just need to explain why \eqref{PlOmega} is a Plancherel formula for our representation $\Omega$, which is equivalent to saying that the forms $J_\pi$ on $\Omega^0$ are positive semi-definite; indeed, once this is established, the Hilbert spaces of the abstract Plancherel formula \eqref{Plabstract} are obtained by completion.
 
From \cite[Theorem 1]{CHH}, we know that $\Omega$ is a tempered representation. There are various equivalent characterizations of this property:
\begin{enumerate}
\item The support of Plancherel measure for $\Omega$ lies in the tempered dual $\hat G^\temp$ of $G$.
\item Every diagonal matrix coefficient of $\Omega$ can be approximated, uniformly on compacta, by diagonal matrix coefficients of the right regular representation $\rho$ of $G$.
\item For every $f\in L^1(G)$, we have an inequality of operator norms:
\begin{equation}\label{operatornorms}
\Vert \Omega(fdg) \Vert \le \Vert \rho(fdg)\Vert.
\end{equation}
\end{enumerate}

For the equivalence, see \cite[Theorems 1.2 and 1.7]{Fell} and \cite[Lemme 1.23]{Eymard}.



The algebra $\mathscr C(G) dg$ of Harish-Chandra Schwartz measures acts on every tempered representation by convolution. I remind the reader an argument for the convergence of the convolution operator $\Omega(hdg)$, for $h\in \mathscr C(G)$ and $\Omega$ a tempered representation: Writing $h\in \mathscr C(G)$ as a limit of $L^1$-functions $f_n$, we see by \eqref{operatornorms} that the limit of the operators $\Omega(f_ndg)$ exists, because the same is true for the limit of the operators $\rho(f_ndg)$.


We are interested in the action of a ring $\mathfrak z(G)$ of ``multipliers'', which lies in the center of a certain completion of the algebra $\mathscr C(G)dg$. This ring can be identified with the algebra $\mathscr C(\mathcal T)$ of Schwartz functions which are supported on a finite number of connected components of the orbifold $\mathcal T$ of tempered parameters. I outline the definitions: The tempered dual $\widehat{G}^\temp$ has a finite-to-one map onto the orbifold $\mathcal T$ whose points are $G$-conjugacy classes of pairs $(L,\tau)$, where $L$ is a Levi subgroup and $\tau$ is the isomorphism class of an irreducible, unitary, discrete-mod-center representation of $L$; the fiber over $(L,\tau)$ is the set of irreducible representations in the unitarily induced representation $I_P^G(\tau)$, where $P$ is a parabolic with Levi subgroup $L$. We let $\mathscr C(\mathcal T)$ be the space of Schwartz functions (i.e.\ smooth functions of rapid decay) on this orbifold which are supported on a finite number of connected components. (Of course, in the non-Archimedean case the notion of rapid decay is irrelevant, since the components are compact. In this case, the ring of multipliers $\mathfrak z(G)$ that I am describing here is the restriction to a finite number of spectral components of the ``tempered Bernstein center'' of Schneider and Zink \cite{SZ}; cf.\ also \cite{DHS}.) 

There is a ring $\mathfrak z(G) \simeq \mathscr C(\mathcal T)$ (we will denote this isomorphism by $z\mapsto \hat z$) acting on any tempered representation $\Omega$ of $G$. Its action on $K$-finite vectors (where $K$ is some maximal compact subgroup) can be described in terms of the Harish-Chandra Schwartz algebra, and it is then a simple exercise on the Plancherel formula to show that this action extends to $\Omega$ (but we will not use this extension). More precisely, let $\hat z\in \mathscr C(\mathcal T)$ be given. For any fixed $K$-type $\tau$, there is, by the Plancherel formula of \cite{HC, WaPl}, a unique element $z^\tau\in \mathscr C(G) dg$ which acts by the scalar $\hat z(\pi)$ on the $(K,\tau)$-type vectors of $\pi\in\widehat{G}^\temp$, and annihilates all other types. (By abuse of notation, we write $\hat z(\pi)$ for the value of $\hat z$ on the image of $\pi$ in $\mathcal T$.) The formal sum:
\begin{equation}
z = \sum_\tau z^\tau
\end{equation}
over all $K$-types is a well-defined operator on the $K$-finite vectors of any tempered representation $\Omega$, and extends uniquely to an operator on $\Omega$ which commutes with the action of $G$.

Returning to the setting of the proposition, for any $v, w\in \Omega^0$, $h\in \mathscr C(G)$ and $x\in G$ we have:
$$ \left< \Omega(x) \Omega(hdg)  v, w\right> = \int_G h(g) \left<\Omega(xg) v, w\right> dg = (\rho(hdg)F)(x),$$
where $F(x)$ is the matrix coefficient $\left<\Omega(x) v, w\right>$. In other words, the matrix coefficient map from $\Omega\otimes\bar\Omega$ to functions on $G$ intertwines the action of the algebra $\mathscr C(G)dg$ on $\Omega$ with its right regular action on functions on $G$ (which, in particular, is convergent on the image of the matrix coefficient map). 

Since the right regular action of $\mathscr C(G)dg$ preserves the space $\mathscr C(G)$, we may assume that the subspace $\Omega^0$ is stable under the action of $\mathscr C(G)dg$ (by replacing it with its span under the action of this algebra). 

Now we return to \eqref{PlOmega}, setting $F(g) = \left<\Omega(g)\Phi,\Phi\right>$. To prove that the hermitian forms $J_\pi$ are positive semi-definite, it is enough to show it for their restrictions to the set of $K$-finite vectors in $\Omega^0$. Thus, assume that $\Phi\in \Omega^0$ is $K$-finite. If $z \in \mathfrak z(G)$ we get:
$$0\le \left< \Omega(z)\Phi, \Omega(z)\Phi\right> = (\rho(z^* z)F)(1)  $$
(where $z^*$ denotes the adjoint operator, which for $fdg \in \mathscr C(G)dg$ is represented by $f^*dg$, with $f^*(g) = \overline{f(g^{-1})}$)
$$ = \int_{\widehat{G}^\temp} |\hat z(\pi)|^2 F_\pi(1) \mu_G(\pi) \mbox{\,\, (by the Plancherel formula for $\mathscr C(G)$)}$$
$$ = \int_{\widehat{G}^\temp} |\hat z(\pi)|^2 J_\pi(\Phi,\Phi) \mu_G(\pi).$$
 
The Plancherel measure is concentrated on the set $\mathcal T^0$ of elements of $\widehat{G}^\temp$ for which the map to $\mathcal T$ is a bijection; indeed, in a family 
of representations parabolically induced from discrete series, with the discrete series varying by unramified twists, generic elements are irreducible. 
Hence, the above integral is really an integral over the orbifold $\mathcal T$. Moreover, the restriction of $F_\pi(1) = J_\pi(\Phi,\Phi)$ to $\mathcal T^0$ extends to a Schwartz function $t\mapsto J_t(\Phi,\Phi)$ on this orbifold, again by the Plancherel formula for $\mathscr C(G)$. If we had $J_t(\Phi,\Phi)<0$ for some $t \in \mathcal T$, and hence in some neighborhood of $t$, we would get a contradiction by choosing $\hat z$ with support in this neighborhood. Thus, $J_t(\Phi,\Phi)\ge 0$ for all $t$. This is enough to establish the Plancherel formula, but we notice that it also implies that $J_\pi(\Phi,\Phi)\ge 0$ for \emph{every} tempered $\pi$: indeed, if the fiber over $t\in \mathcal T$ contains a finite number of elements $\pi_1 ,\dots, \pi_k$, and $\pi_i$ appears with multiplicity $m_i\ge 1$ in the corresponding induced representation, we have just shown that the sum $\sum_{i=1}^k m_i J_{\pi_i}$ is a positive semi-definite hermitian form, but since  the representations $\pi_i$ are inequivalent, each of the summands has to be positive semi-definite.
 
\end{proof}

Combining this with Proposition \ref{inHCS}, Theorem \ref{localtheorem} follows. We also get as a corollary the following result, which generalizes a theorem of Hongyu He \cite{He}, answering a question of Jian-Shu Li \cite{Li} in the case where $\pi_2$ is tempered:

\begin{corollary}
The hermitian forms $J_\pi^\Planch$ of \eqref{JPlanch} are positive semi-definite, for every tempered representation $\pi_2$ of $G_2$.
\end{corollary}

\section{Euler factorization} \label{sec:global}

\subsection{Global conjecture and theorem}

With the spaces $V$ and $W$ defined over a global field $k$ (or its quadratic extension $E$, in the unitary case), and the additive adele class character $\psi$, there is a central extension:
\begin{equation}\label{globalcover} 1\to \CC^1 \to \widetilde\Sp(V\otimes W)(\adele) \to \Sp(V\otimes W)(\adele)\to 1,
\end{equation}
equipped with a canonical splitting $\Sp(V\otimes W)(k) \hookrightarrow \widetilde\Sp(V\otimes W)(\adele)$, a genuine, unitary representation $\oomega$ and an automorphic (i.e.\ $\Sp(V\otimes W)(k)$-invariant) functional:
\begin{equation}\label{autfnl}
\mathcal I: \oomega^\infty \to \CC.
\end{equation}

Moreover, we may fix a restricted tensor product decomposition, assumed unitary:
\begin{equation}\label{omegafactor}
\oomega^\infty = \bigotimes'_v \oomega_v^\infty,
\end{equation}
such that the factor $\oomega_v$ is the oscillator representation for $\widetilde\Sp(V\otimes W)(k_v)$, the preimage of $\Sp(V\otimes W)(k_v)$ in \eqref{globalcover}. Choosing a Lagrangian subspace $Y\subset V\otimes W$, the representation $\oomega$ admits a Schr\"odinger model with 
$$ \oomega^\infty = \mathcal S(Y^*(\adele)),$$
the canonical functional $\mathcal I$ being equal to the functional:
\begin{equation}
\mathcal S(Y^*(\adele)) \ni \Phi \mapsto \sum_{\gamma\in Y^*(k)} \Phi(\gamma),
\end{equation}
and the factorization \eqref{omegafactor} can be taken to be an analogous factorization for the Schwartz space $\mathcal S(Y^*(\adele))$, with $\oomega_v$ realized on its Schr\"odinger model on $\mathcal S(Y^*(k_v))$. Measures on adelic vector spaces will always be taken to be Tamagawa measures, and we fix Euler factorizations of those which on $k_v$ or on the $k_v$-points of a symplectic space are self-dual with respect to the local factor $\psi_v$ of the character $\psi$.

We now restrict the global oscillator representation $\oomega$ to the dual pair $G(\adele) = G_1(\adele)\times G_2(\adele)$, where by $G_i(\adele)$ we denote the $\CC^1$-cover obtained by restricting \eqref{globalcover} to the appropriate subgroup of $\Sp(V\otimes W)(\adele)$. The canonical splitting of the cover over $\Sp(V\otimes W)(k)$ induces a splitting of $G_i(\adele)$ over the $k$-points of the corresponding classical group, and we will denote the corresponding subgroup of $G_i(\adele)$ by $G_i(k)$.

Let $[G_i] =G_i(k)\backslash G_i(\adele)$. The space $[G_i]/\CC^1$ is the automorphic quotient of a classical group, and we endow it with Tamagawa measure. We fix a factorization of this measure into measures for the local groups $G_{i,v}/\CC^1$. The notation $L^2([G_i])$ will denote the space of \emph{genuine} functions on $[G_i]$, endowed with the $L^2$-inner product obtained by integrating against Tamagawa measure on $[G_i]/\CC^1$.  Similarly, by ``automorphic representation'' for $G_i$ we will mean, throughout, a \emph{genuine} automorphic representation.

Let $\pi_1,\pi_2$ be genuine, discrete automorphic representations of $G_1$, $G_2$, considered as subspaces of $L^2([G_i])$ (not just as abstract adelic representations) where $[G_i]=G_i(k)\backslash G_i(\adele)$. Set $\pi= \pi_1\otimes\pi_2$, and define the theta pairing:
\begin{equation}
\mathcal P_\pi^\Aut : \bar\pi^\infty \otimes\oomega^\infty\to \CC
\end{equation}
by:
\begin{equation}
 \bar\pi^\infty \otimes\oomega^\infty \ni \overline{(\varphi_1\otimes\varphi_2)}\otimes\Phi \mapsto \int_{[G_1]/\CC^1 \times [G_2]/\CC^1} \overline{\varphi_1(g_1) \varphi_2(g_2)} \mathcal I(\oomega(g_1,g_2)\Phi) d(g_1,g_2),
\end{equation}
whenever the integral above is absolutely convergent. This is of course nothing else than the integral of $\overline{\varphi_1}$ against the theta lift of $\overline{\varphi_2}$ with respect to $\Phi$ (or vice versa): 

\begin{equation}
\mathcal P_\pi^\Aut \left(\overline{(\varphi_1\otimes\varphi_2)}\otimes\Phi \right) = \int_{[G_1]/\CC^1} \overline{\varphi_1(g_1)} \theta^\Aut(\Phi,\overline{\varphi_2})(g_1) dg_1,
\end{equation}
where
\begin{equation}\label{thetaAut}
\theta^\Aut(\Phi,\overline{\varphi_2})(g_1) = \int_{[G_2]/\CC^1} \overline{\varphi_2}(g_2)\mathcal I(\oomega(g_1,g_2)\Phi)  dg_2.
\end{equation}
(The image of the theta lift of elements of $\overline{\pi_2}$ will be denoted by $\theta^\Aut(\pi_2)$, for compatibility with the local notation.) We will eventually work with $\pi_2$ cuspidal, in which case, as we will recall, $\theta^\Aut(\Phi,\overline{\varphi_2})(g_1)$ is in the discrete automorphic spectrum of $[G_1]$ (possibly zero); in particular, the pairings above are defined for every $\pi_1$, and non-zero iff $\pi_1$ is not orthogonal to the image of $\theta^\Aut$.

The pairing $\mathcal P^\Aut_\pi$ is the analog of the following ``period'' functional when $X$ is a quasi-affine homogeneous spherical variety for a group $G$ over $k$:
\begin{equation}
\bar\pi^\infty \otimes \mathcal S(X(\adele)) \ni \varphi \otimes \Phi \mapsto \int_{[G]} \overline{\varphi(g)} \sum_{\gamma\in X(k)} \Phi(\gamma g) dg
\end{equation}
which, when $X(\adele)= H(\adele)\backslash G(\adele)$ unfolds to the more familiar period integrals of automorphic forms in $\bar\pi$ over $[H]$.

The pairing $\mathcal P_\pi^\Aut$ is an element of the space $\Hom_G (\bar\pi^\infty\otimes\oomega^\infty,\CC)$, which is a restricted tensor product of the (at most one-dimensional) spaces $\Hom_{G(k_v)} (\bar\pi^\infty_v\otimes\oomega^\infty_v,\CC)$, after we fix a \emph{unitary} isomorphism with a restricted tensor product of local unitary representations:
\begin{equation} 
\pi^\infty \simeq \bigotimes'_v \pi^\infty_v.
\end{equation}

Our goal is to describe the Euler factorization of $\mathcal P_\pi^\Aut$, or rather of the square of its absolute value, in terms of the local Plancherel formula. It will be more convenient to dualize. Hence, let:
\begin{equation}
J_\pi^\Aut : \oomega^\infty\otimes\bar\oomega^\infty\to \pi\otimes\bar\pi \to\CC
\end{equation}
be given by:
\begin{equation}
J_\pi^\Aut (\Phi_1\otimes\overline{\Phi_2}) = \sum_\varphi \mathcal P^\Aut_\pi (\bar\varphi\otimes\Phi_1) \cdot \overline{\mathcal P^\Aut_\pi (\bar\varphi\otimes\Phi_2)},
\end{equation}
the sum ranging over an orthonormal basis of $\pi$. We will call this the \emph{global}, or \emph{automorphic relative character} of Howe duality.

For most of the groups that we are considering, the notion of global Arthur parameters is meaningful, by the work of Arthur and others \cite{Abook,Mok,KMSW}. In particular, we know that the discrete automorphic specturm of $G_1$ and $G_2$ is partitioned in ``packets'' parametrized by ``equivalence classes of global Arthur parameters'' which are symbolically written:
\begin{equation}\psi: \mathcal L_k\times \SL_2 \to \LG_i,\end{equation}
(where $\mathcal L_k$ is supposed to denote the ``global Langlands group'' of the field, an extension of the global Weil group) and whose real meaning is explained in terms of automorphic representations of general linear groups in \cite{Abook}[\S 1]. We will assume this parametrization to be known in all cases, which is work in progress, cf.\ \cite{WWL}, in order to formulate our global conjecture.

\begin{remark}
Of course, our groups are non-standard, and to identify them with classical groups or double metaplectic covers thereof, we need to choose idele class characters as in the local discussion of \S \ref{ssmetaplectic}. It is clear that, once the Arthur parametrization is known for the corresponding classical or metaplectic group, it gives rise in a \emph{canonical} way to an analogous parametrization of genuine automorphic representations for our $\CC^1$-covering group in terms of parameters into its non-standard $L$-group that we attached in \S \ref{ssLgroups}. (For metaplectic double covers, we recall that this parametrization is only canonical given the choice of additive character $\psi$.)
\end{remark}

We will start by formulating a global conjecture, the analog of \cite{SV}[Con\-jecture 17.4.1].  
We will restrict ourselves to the ``most tempered'' part of the distinguished spectrum, which should correspond to discrete ``global Arthur parameters'' of the form \eqref{composition} (with the Weil-Deligne group replaced by $\mathcal L_k$). Thus, call a global Arthur parameter for $G=G_1\times G_2$ ``$\oomega$-distinguished'' if it is of the form:

\begin{equation}\label{composition-global} \mathcal L_k \times \SL_2 \xrightarrow{\phi \times \Id} \LG_\oomega \times \SL_2 \xrightarrow{\gamma} {^L(G_1 \times G_2)},\end{equation}
where the canonical map $\gamma$ is that of \eqref{canonicalmorphism}, and $\phi$ is discrete (i.e.\ does not ``factor through a Levi subgroup'' -- again, this has to be understood as in Arthur's book).


Attached to such a parameter, there is (by the aforementioned work of Arthur and others, some in progress) a canonical subspace $\mathcal A_\phi$ of $L^2([G])$. For the groups under consideration, this space is multiplicity-free, i.e.\ every irreducible adelic representation in this space appears with multipliplicity one. 
Notice that for the (disconnected) even orthogonal group this is not explicitly stated in \cite{Abook}, but can be inferred from Theorem 1.5.2 in that monograph. It has also been established by means of the theta correspondence by Atobe-Gan \cite{AG-orthogonal}.

\begin{conjecture}
Consider an $\oomega$-distinguished global Arthur parameter induced from a discrete parameter $\phi:\mathcal L_k\to \LG_2$ as above. Let $\pi=\pi_1\otimes\pi_2\subset \mathcal A_\phi$, and assume (as expected by the generalized Ramanujan conjecture) that $\pi_2$ is tempered. Then there is a rational number $q$ such that:
\begin{equation}\label{Euler} J_\pi^\Aut = q \prod_v^* J_{\pi_v}^\Planch. \end{equation}
Here $\pi_v = \pi_{1,v} \otimes \pi_{2,v}$ with $\pi_{1,v}$ the theta lift of $\pi_{2,v}$, and $J_{\pi_v}^\Planch$ denotes the hermitian forms  of the local Plancherel formula of Theorem \ref{localtheorem} (possibly zero). The regularized Euler product is understood as follows: For almost every place, the Euler factor will be equal to a local unramified $L$-factor:
\begin{equation}
\frac{L_v(s_0 + \frac{1}{2} , \pi_v)}{d_v (s_0)},
\end{equation}
where ${d_v (s)}$ is an explicit product of Hecke $L$-factors and $s_0=\frac{m-d(n)}{2}\ge 0$ as in \eqref{s0}. We thus understand the Euler product as the partial $L$-value:
\begin{equation}\label{partialL}
\frac{L^S(s_0 + \frac{1}{2} , \pi)}{d^S (s_0)},
\end{equation}
where $S$ is a sufficiently large set finite set of places, times the above product over $v\in S$.
\end{conjecture}

\begin{remark}
The $L$-function appearing in the numerator is the ``standard $L$-function'' for parameters into our non-standard $L$-groups of \S \ref{ssLgroups}. It is defined as in the ``standard'' case: When the $L$-group is a subgroup of $O(\CC)\times \Gal(\bar k/k)$ or $\Sp(\CC)\times\Gal(\bar k/k)$ (we omit the dimension from the notation), the standard representation is the standard representation of $O(\CC)$, resp.\ $\Sp(\CC)$. In the unitary case, where the connected component of the $L$-group is $\GL(\CC)$, the standard representation is obtained by base change from $k$ to $E$. The reader can easily see that our non-standard $L$-groups spare us the need to include the character $\chi_V$ in formulas as in \cite{GQT}. The abelian $L$-factor $d_v(s_0)$ appears in the evaluation of local zeta integrals, cf.\ \cite[\S 11.6]{GQT} and has been computed in \cite[p.334, Remark 3]{LR}. I leave it to the reader to explicate the precise meaning of the abelian factors ${d_v (s_0)}$ in our setting, since this is somewhat orthogonal to the goal of this paper.
\end{remark}

Let $\mathcal A_{\phi,i}$ be the projection of $\mathcal A_\phi$ to $L^2([G_i])$. If we assume that the global theta lift of an irreducible $\pi_2\subset \mathcal A_{\phi,2}$ lies in $\mathcal A_{\phi,1}$, which is the global version of Adams' conjecture, and restrict our attention to the case where both $\pi_2$ and $\theta^\Aut(\pi_2)$ are (zero or) cuspidal, the conjecture above is a consequence of the following:

\begin{theorem}\label{globaltheorem} 
Assume that $\pi = \theta^\Aut(\pi_2) \otimes \pi_2$ is cuspidal (or zero), with $\pi_2$ tempered. Exclude the case when $W$ is symplectic of dimension $n$ and $V$ is \emph{split} orthogonal of dimension $2(n+1)$. Then $\pi =0$ if and only if the right hand side of \eqref{Euler} is zero, and in any case formula \eqref{Euler} is true on the space of $\pi$, with $q=[E:k]$.
\end{theorem}

The statement is just a reinterpretation of known versions of the Rallis inner product formula, as we will see in the next section.

\begin{remarks}\begin{enumerate}
\item
The statement means that $\theta^\Aut(\pi_2)$ is zero if and only if the partial $L$-function is zero at $s_0$, or some of the local factors $J_{\pi_v}^\Planch$ are zero.

\item
Temperedness ensures that the local Euler factors $J_{\pi_v}^\Planch$ are well-defined. These factors also turn out to be meromorphic in the parameters of the representation, and the formulas of \cite{GQT, Yamana} extend the above result off the tempered spectrum of $G_2$. 

\item The condition that $\theta^\Aut(\pi_2)$ is cuspidal is pretty restrictive: it only happens once in every Witt tower. It would be desirable to verify the conjecture also in the case where $\pi_2$ is discrete but not cuspidal. I do not do this in this paper.
\end{enumerate}
\end{remarks}

The goal of the rest of this section will be to prove this theorem.

\subsection{Reinterpretation of the Rallis inner product formula}

\subsubsection{Different ranges for the theta correspondence}

For the purposes of applying the Siegel-Weil theorem and obtaining the Rallis inner product formula, one distinguishes the following cases for a dual pair $(G_1, G_2)$ as in Table \eqref{table} -- without, here, necessarily assuming that $m\ge d(n)$:

First of all, let $r$ denote the Witt index of $V$ (the dimension of its maximal isotropic subspace). 

\begin{itemize}
\item When $r=0$ or $m-r > d(n)$, we are in the \emph{convergent range}.

\medskip

Assume now that this is not the case.

\item The case $m<d(n)$ is the \emph{first term range}.

\item The case $m=d(n)$ is the \emph{boundary case}.

\item The case $d(n)<m\le 2d(n)$ is the \emph{second term range}.
\end{itemize}

Having assumed that $m\ge d(n)$ we are either in the boundary or second term range, or in the convergent range.
For the second term range, we also need to impose an extra condition imposed by Kudla-Rallis: 
\begin{equation}
r\le n.
\end{equation}
The only case we are excluding by this condition is when $\epsilon =1$, $m=2d(n)$ and $V$ is split, i.e.\ the lift from $G_2=\Sp_n$ to the split $G_1=O_{2(n+1)}$. It is a technical condition that should not be necessary, but needs different techniques, see the discussion in \cite[\S 3.5]{GQT}. 

\subsubsection{Statement of the Rallis inner product formula}

Denote, as before, by $\WW$ the $(-\epsilon)$-hermitian space $W\oplus (-W)$, and by $\tilde\oomega$ the oscillator representation of the adelic covering group $\widetilde\Sp(V\otimes\WW)(\adele)$. By the isomorphism \eqref{double}, any pair of vectors $\Phi_1,\Phi_2\in\oomega$ give rise to a vector in $\tilde\oomega$, which we will denote just by $\Phi_1\otimes \overline{\Phi_2}$. 

Recall also the morphism \eqref{toinduced}, from the oscillator representation attached to $V\otimes\WW$ to a degenerate principal series. This morphism was obtained by realizing the local oscillator representation $\tilde\oomega_v$ on the space $L^2(\ell^*(k_v))$, where $\ell$ was the Lagrangian subspace $V\otimes W$, embedded diagonally. In \S \ref{sscomparison} we described ways to achieve this realization, by choosing a Haar measure on $\ell^*(k_v)$ and a splitting of the natural map $V\otimes\WW \to V\otimes\WW/\ell = \ell^*$, and we described canonical local choices for those data. These choices do not matter globally (any global choice of splitting, together with Tamagawa Haar measure, will produce the same functional), but we work with them throughout in order to fix an identification of the local factors of $\tilde\oomega$ as in \eqref{omegafactor} with $L^2(\ell^*(k_v))$, and a factorization of the global functional $\ev_0$ with the corresponding product of local functionals.

We will denote the image of a vector $\Phi \in \tilde\oomega^\infty$ by $f_\Phi \in \Ind_{\SS(\adele)}^{\widetilde\Sp(V\otimes\WW)(\adele)}(\chi)$, and based on the isomorphisms just recalled it comes with an Euler factorization, induced from \eqref{omegafactor}:
\begin{equation}
f_\Phi = \bigotimes_v f_{\Phi_v}.
\end{equation}

\begin{theorem} \label{RIPtheorem}
Suppose that $m\ge d(n)$ and $r\le n$. 

Let $\pi_2$ be a cuspidal, genuine automorphic representation of $G_2$ and consider the global theta lift $\pi_1=\theta^\Aut(\pi_2)$  to $G_1$. Then:

\begin{enumerate}
\item $\pi_1$, if non-zero, is discrete (i.e., it lies in $L^2([G_1])$). It is cuspidal if and only if $V$ is anisotropic, or the global theta lift to $G_1^{-2}$ is zero, where $G_1^{-2}$ is constructed like $G_1$, replacing $V$ by an $\epsilon$-hermitian space $V'$ such that $V$ is the orthogonal sum of $V'$ and a split two-dimensional $\epsilon$-hermitian subspace.
\item If $\pi_1$ is cuspidal then for $\varphi_1, \varphi_2\in \pi_2$ we have: 
\begin{equation}\label{RIP}
\left<\theta^\Aut(\overline{\varphi_1},\Phi_1),\theta^\Aut(\overline{\varphi_2},\Phi_2)\right>_{L^2[G_1]} = [E:k] \prod^*_v Z_v (\varphi_{1,v}, \varphi_{2,v}, f_{\Phi_{1,v}\otimes\overline{\Phi_{2,v}}})
\end{equation} where $s_0=\frac{m-d(n)}{2}$, and $Z_v$ is the local zeta integral \eqref{localzeta} of the doubling method:
\begin{equation}\label{localzeta2}
Z_v (\varphi_{1,v}, \varphi_{2,v}, f_{\Phi_v}) = \int_{G_{2,v}/\CC^1} f_{\Phi_v}(g, 1) \overline{\left<  \pi_v(g)\varphi_{1,v}, \varphi_{2,v}\right>} dg.
\end{equation}
\end{enumerate}
\end{theorem}

\begin{proof}
For the discreteness statement, see \cite{Rallis-Howe}; also \cite[Proposition 10.1]{Yamana} -- notice that condition (4) of this proposition is satisfied because we are assuming $m\ge d(n)$.

Under the assumption that $\pi_1$ is cuspidal, \eqref{RIP} is \cite[Theorem 11.4]{GQT}, except for the boundary case where I point the reader to \cite[Lemma 10.1]{Yamana}. 
\end{proof}

\subsubsection{Reinterpretation in terms of the Plancherel formula}

We will now reinterpret the Rallis inner product formula to prove Theorem \ref{globaltheorem}.

By Lemma \ref{ev0}, we have $f_{\Phi_{1,v}\otimes\overline{\Phi_{2,v}}} (g,1) = \left<\oomega(g) \Phi_{1,v},\Phi_{2,v}\right>_\oomega$ for $g\in G_2(\adele)$. Hence, the local zeta integral \eqref{localzeta2}, for $\Phi_v = \Phi_{1,v}\otimes\overline{\Phi_{2,v}}$ can be written:
\begin{equation}
Z_v (\varphi_{1,v}, \varphi_{2,v}, f_{\Phi_{1,v}\otimes\overline{\Phi_{2,v}}}) =  \int_{G_{2,v}/\CC^1} \left<\oomega(g) \Phi_{1,v},\Phi_{2,v}\right>_\oomega \overline{\left<  \pi_v(g)\varphi_{1,v}, \varphi_{2,v}\right>} dg.
\end{equation}

By Theorem \ref{localtheorem}, these are the morphisms of the local Plancherel formula with measure $\mu_{G_{2,v}}$, the Plancherel measure for $G_{2,v}$. More precisely, in the notation used here:
\begin{equation} \label{PlanchZeta}
J_{\pi_v}^\Planch (\Phi_{1,v}\otimes\overline{\Phi_{2,v})} = \sum_{\varphi_{2,v}} Z_v (\varphi_{2,v}, \varphi_{2,v}, f_{\Phi_{1,v}\otimes\overline{\Phi_{2,v}}}),
\end{equation}
the sum ranging over an orthonormal basis for $\pi_{2,v}$.

Thus, 
$$J_\pi^\Aut (\Phi_1\otimes\overline{\Phi_2}) = \sum_\varphi \mathcal P^\Aut_\pi (\bar\varphi\otimes\Phi_1) \cdot \overline{\mathcal P^\Aut_\pi (\bar\varphi\otimes\Phi_2)}$$
(sum over $\varphi = \varphi_1 \otimes \varphi_2$ in an orthonormal basis of $\pi = \pi_1 \otimes\pi_2$)
$$ = \sum_{\varphi_1 \otimes \varphi_2}  \int_{[G_1]/\CC^1} \overline{\varphi_1(g_1)} \theta^\Aut(\Phi_1,\overline{\varphi_2})(g_1) dg_1 \cdot \overline{ \int_{[G_1]/\CC^1} \overline{\varphi_1(g_1)} \theta^\Aut(\Phi_2,\overline{\varphi_2})(g_1) dg_1,} =$$
$$ = \sum_{\varphi_2} \left<\theta^\Aut(\overline{\varphi_2},\Phi_1),\theta^\Aut(\overline{\varphi_2},\Phi_2)\right>_{L^2[G_1]} =$$ 
$$ = [E:k] \prod_v^* \sum_{\varphi_{2,v}} Z_v (\varphi_{2,v}, \varphi_{2,v}, f_{\Phi_{1,v}\otimes\overline{\Phi_{2,v}}}) \mbox{ \, \, (by Theorem \ref{RIPtheorem})}$$
$$ = [E:k] \prod_v^* J_{\pi_v}^\Planch (\Phi_{1,v}\otimes\overline{\Phi_{2,v})} \mbox{ \,\, (by \eqref{PlanchZeta})},$$
as asserted in Theorem \ref{globaltheorem}.

\bibliographystyle{alphaurl}
\bibliography{howe}

\end{document}